\documentclass[12pt]{article}
\usepackage{amsmath,amssymb,amsfonts,amsthm,graphicx,color}
\usepackage{bbm}
\usepackage{mathrsfs}
\usepackage{hyperref}
\usepackage{mathtools}
\usepackage{graphicx,type1cm,eso-pic,color}
\usepackage[margin=1in]{geometry}
\allowdisplaybreaks
\title{Blow-up estimates for a system of semilinear SPDEs\\ driven by mixed fractional Brownian motions}
\author{S. Sankar${}^{1}$, Manil T. Mohan${}^{2}$,  and S. Karthikeyan${}^{1,}\thanks{Corresponding author email: karthi@periyaruniversity.ac.in}$\\	\footnotesize{$^1$Department of Mathematics, Periyar University, Salem 636 011, India}\\ \footnotesize{$^2$Department of Mathematics, Indian Institute of Technology Roorkee, Roorkee 247 667, India}}
\date{}
\allowdisplaybreaks

\usepackage{xcolor}
\allowdisplaybreaks
\usepackage[pagewise]{lineno}
\usepackage{graphicx,eurosym}
\usepackage{hyperref}
\usepackage{mathtools}
\colorlet{darkblue}{blue!50!black}
\hypersetup{
colorlinks,%
citecolor=red,%
filecolor=red,%
linkcolor=blue,%
urlcolor=blue,%
pdfnewwindow=true,%
pdfstartview={FitH}}

\let\originalleft\left
\let\originalright\right
\renewcommand{\left}{\mathopen{}\mathclose\bgroup\originalleft}
\renewcommand{\right}{\aftergroup\egroup\originalright}

\begin{document}
	\maketitle \setcounter{page}{1} \numberwithin{equation}{section}
	\newtheorem{theorem}{Theorem}[section]
	\newtheorem{assumption}{Assumption}
	\newtheorem{lemma}{Lemma}[section]
	\newtheorem{Pro}{Proposition}[section]
	\newtheorem{Ass}{Assumption}[section]
	\newtheorem{Def}{Definition}[section]
	\newtheorem{Rem}{Remark}[section]
	\newtheorem{corollary}{Corollary}[section]
	\newtheorem{proposition}{Proposition}[section] 
	\newtheorem{app}{Appendix:}
	\newtheorem{ack}{Acknowledgement:}

   \begin{abstract}
   In this paper, we obtain the existence and finite-time blow-up for the solution to a system of semilinear stochastic partial differential equations driven by a combination of Brownian and fractional Brownian motions. Under suitable assumptions,  lower and upper bounds for the finite-time blow-up solution are obtained. We provide sufficient conditions for the existence of a global weak solution to the system. Further, a lower bound for the probability of the finite-time blow-up solution of the considered system is provided by using Malliavin calculus.\\
		
		
   \end{abstract}

	
	\section{Introduction}
    Fujita \cite{fuji1966, Fuji1968}, in this pioneering work, proved that the semilinear heat equation 
	\begin{equation*}
	\frac{\partial u(t,x)}{\partial t}=\Delta u(t,x)+u^{1+\alpha}(t,x), \ x\in D,
	\end{equation*}
    defined on a smooth bounded  domain $D\subset\mathbb{R}^d$, $d\geq 1$ with the Dirichlet boundary condition, where $\alpha>0$ is a constant, explodes in finite-time for all non-negative initial data $f(x)\in L^2(D)$ satisfying $$\int_D f(x)\psi(x)dx>\lambda^{\frac{1}{\alpha}},$$ $\lambda>0$ being the first eigenvalue of $-\Delta$ on $D$ and $\psi>0$, the corresponding normalized eigenfunction with $\|\psi\|_{L^1}=1$.

   The problem of existence of global positive solutions of semilinear stochastic partial differential equations (SPDEs) has been extensively investigated \cite{chow09, chow11}.  \textcolor{red}{ Dozzi and L\'{o}pez-Mimbela} \cite{doz2010} investigated the  blow-up problem for the stochastic analog of the above deterministic problem, namely
   \begin{equation}\label{d1}
   \left\{
   \begin{aligned}
   du(t,x)&=\left[\Delta u(t,x)+G(u(t,x)) \right]dt+\kappa u(t,x)dW(t),\ t>0,\\
   u(0,x)&=f(x) \geq 0, \ x\in D, \\
	u(t,x)&=0, \ t \geq 0, \ x \in \partial D,
	\end{aligned}
	\right.
	\end{equation}
	where $D\subset \mathbb R^d$ is a smooth bounded domain, $G:\mathbb{R} \rightarrow \mathbb{R}^{+}$ is locally Lipschitz and satisfies \begin{align}\label{1.2} G(z)\geq Cz^{1+\beta} \ \mbox{ for all } \ z>0,\end{align}  $C,\kappa$ and $\beta$ are given positive numbers and $W(\cdot)$, a standard one-dimensional Brownian motion defined on some probability space $\left( \Omega, \mathcal{F}, \mathbb{P} \right)$. By transforming the above equation to a  random partial differential equation, the bounds for the explosion time are estimated in terms of exponential functionals of $W(\cdot)$. For the case, $G(u)=u^{1+\beta},\ \beta>0$, the upper and lower bounds for the explosion times are estimated in \cite{car2013} by using the formula derived by Yor \cite{yor2001}.     

	The impact of Gaussian noises on the finite-time of blow-up of a nonlinear SPDE 
	$$du(t,x)=[\Delta  u(t,x)+G(u(t,x))]dt+k_1u(t,x)dW_{1}(t)+k_2u(t,x)dW_2(t), $$
	with Dirichlet boundary condition, where  $\{ W_{1}(t),W_{2}(t) \}_{t \geq 0} $ is a two-dimensional Brownian motion was investigated by Niu and Xie \cite{niu2012}.  Dozzi \textcolor{red}{et al.,} \cite{doz2013} extended the results to a system of semilinear SPDEs of the form:
	\begin{equation*}
	\left\{
	\begin{aligned}
	du_{1}(t,x)&=\left[ \left(\Delta+V_{1}\right)u_{1}(t,x)+u^{p}_{2}(t,x) \right]dt+k_{1}u_{1}(t,x)dW(t) ,\nonumber\\
	du_{2}(t,x)&=\left[ \left(\Delta+V_{2}\right)u_{2}(t,x)+u^{q}_{1}(t,x) \right]dt+k_{2}u_{2}(t,x)dW(t), \ x\in D,\ t>0,
	\end{aligned}
	\right.
	\end{equation*}
	with the Dirichlet boundary data $$u_i(0,x)= f_i(x)\geq 0,\ x\in D \mbox{ and } u_i(t,x)= 0,\ t\geq 0,\ x\in \partial D.$$ Here $p\geq q>1$ are constants, $D \subset \mathbb{R}^{d}$ is a bounded smooth domain, $V_{i}>0$ and $k_{i}\neq 0$ are constants, $i=1,2$.  In a recent work, Sankar et al.,  \cite{skm1} obtained  lower and upper bounds for the finite-time blow-up to a more general system of semilinear SPDEs.  

	\textcolor{red}{Dozzi et al.,  \cite{dozfrac2010} derived lower and upper bounds for the explosion time to the equation} 
	\begin{equation}\label{d}
	\begin{aligned}
	du(t,x)&=\left[\Delta u(t,x)+\gamma u(t,x)+G(u(t,x)) \right]dt+ku(t,x)dB^{H}(t),\\
	\end{aligned}
	\end{equation}
	defined on a smooth bounded domain $D\subset \mathbb R^d$ with Dirichlet boundary data for real numbers $\gamma$ and $k$. Here $G:\mathbb{R} \rightarrow \mathbb{R}^{+}$ is locally Lipschitz and satisfies $G(z)\geq Cz^{1+\beta}, \ \mbox{for all} \ z>0$ and positive real numbers $\beta$ and $C$. 
	Here  $\{B^{H}(t)\}_{t\geq 0}$ is a one-dimensional fractional Brownian motion  with Hurst index $H>\frac{1}{2}$. Further, sufficient conditions are established for blow-up in finite-time and for the existence of a global solution of the equation \eqref{d}.
 	Recently,  Dung \cite{dung} provided lower and upper bounds for the probability of the finite-time blow-up of solution of the equation \eqref{d} in terms of Hurst parameter $H,$ where $H \in (0,1)$. However, there  are  limited results available in the literature for the blow-up of solution of SPDEs driven by fractional Brownian motions. Very recently, Dozzi et al., \cite{doz2023} estimated the blow-up behaviour and the probability of finite-time blow-up  of  a semilinear SPDE driven by mixed Brownian motion and fractional Brownian motions.  Motivated by the above facts, we investigate the finite-time blow-up and probability characteristics of a system of semilinear SPDEs driven by mixture of Brownian motion and fractional Brownian motion (fBm).
 
 

	\section{Problem Formulation and Preliminaries}
	Let us consider a  system of semilinear SPDEs driven by mixture of Brownian motion and fractional Brownian motion (fBm) of the form:
	\begin{equation}\label{b1}
	\left\{
	\begin{aligned}
	du_{1}(t,x)&=\left[\Delta u_{1}(t,x)+\gamma_{1}u_{1}(t,x)+u^{1+\beta_{1}}_{2}(t,x) \right]dt \\
	&\qquad \ \ +k_{11}u_{1}(t,x)dW(t)+k_{12}u_{1}(t,x)dB^{H}(t), \\
	du_{2}(t,x)&=\left[\Delta u_{2}(t,x)+\gamma_{2}u_{2}(t,x)+u^{1+\beta_{2}}_{1}(t,x) \right]dt \\
	&\qquad \ \ +k_{21}u_{2}(t,x)dW(t)+k_{22}u_{2}(t,x)dB^{H}(t), \\
	\end{aligned}
	\right.
	\end{equation}
	for $x \in D,\ t>0$, along with the Dirichlet boundary conditions
	\begin{equation}\label{b2}
	\left\{
	\begin{array}{ll}
	u_{i}(0,x)=f_{i}(x),  &x \in D, \\
	u_{i}(t,x)=0, \ \ & x \in \partial D,\ t\geq 0,\  i=1,2,
	\end{array}
	\right.
	\end{equation}
	where $D \subset \mathbb{R}^{d}$ is a bounded domain with smooth boundary $\partial D$.	Here $\beta_{1}\geq \beta_{2}>0,\ \gamma_{i}>0$ and $k_{ij}\geq 0, i,j=1,2$ are constants. The non-negative functions $f_{1},f_2$ are of class $C^{2}$ and not identically zero. Here $\{W(t)\}_{t \geq 0} $ and $\{B^{H}(t)\}_{t \geq 0}$  are standard one-dimensional Brownian and fractional Brownian motions (fBm) respectively with  Hurst index $\frac{1}{2}<H<1$ and are defined on the same filtered probability space $\left( \Omega, \mathcal{F}, (\mathcal{F}_{t})_{t \geq 0}, \mathbb{P} \right)$. In the case of $H=\frac{1}{2}$ and $\gamma_{1}=\gamma_{2}=0$ in \eqref{b1}, lower and upper bounds for the blow-up times of solution are obtained in \cite{li}. 

	fBm arises in various stochastic phenomena characterized by rough external forces. In contrast to standard Brownian motion, fBm is neither a semimartingale nor a Markov process, which makes it challenging to apply classical theory of stochastic integration. Furthermore, considering their  analytic and probabilistic characteristics, the existence of both fractional and Brownian motion of the system \eqref{b1}-\eqref{b2} models various facets of the solution representing evolution of random growth over time.

	Let us recall the notion of weak and mild solutions $u=(u_{1},u_{2})^{\top}$ of the system \eqref{b1}-\eqref{b2}.

	\begin{Def}{\bf (Weak solution)} Let $0\leq \tau= \tau(\omega) \leq  +\infty$ be a stopping time. A continuous $\{\mathcal{F}_{t}\}_{t \geq 0}$-adapted random field $u=(u_1,u_2)^{\top}=\left\lbrace (u_{1}(t,x),u_2(t,x))^{\top}: \ t\geq 0,\ x \in D \right\rbrace $ is a {\it weak solution} of the system  \eqref{b1}-\eqref{b2} on the interval $(0,\tau)$ provided 
	\begin{align*} 
	\int_{D}u_{i}\left(t,x\right)\varphi_{i}(x)dx&=\int_{D}u_{i}\left(0,x\right)\varphi_{i}(x)dx+\int_{0}^{t}\int_{D}u_{i}\left(s,x\right)\Delta \varphi_{i}(x)dxds+\gamma_{i}\int_{0}^{t}\int_{D}u_{i}\left(s,x\right)\varphi_{i}(x)dxds\nonumber\\
	&\quad+\int_{0}^{t}\int_{D} u^{1+\beta_{i}}_{j}\left(s,x\right)\varphi_{i}(x)dxds+k_{i1}\int_{0}^{t}\int_{D}u_{i}\left(s,x\right)\varphi_{i}(x)dxdW(s)\nonumber\\
	&\quad+k_{i2}\int_{0}^{t}\int_{D}u_{i}\left(s,x\right)\varphi_{i}(x)dxdB^{H}(s),\ \mathbb{P}\mbox{-a.s.},
	\end{align*} 	
	holds for every $\varphi_i\in C^\infty_0(D),\ i=1,2, \ j\in \left\{ 1,2 \right\} / \left\{i\right\}$ and for each $t\in(0,\tau)$.
	\end{Def}

	\begin{Def}{\bf (Mild solution)}
	For any non-negative bounded measurable initial data $f=(f_{1}, f_{2})^{\top}$, if there exists a number $0<\tau=\tau(\omega)\leq \infty$, we say that a continuous $\{\mathcal{F}_{t}\}_{t \geq 0}$-adapted random field, $u=(u_1,u_2)^{\top}=\left\lbrace (u_{1}(t,x),u_2(t,x))^{\top}: \ t\geq 0,\ x \in D \right\rbrace $ is a \emph{mild solution} of the system \eqref{b1}-\eqref{b2} on the interval $(0,\tau)$ if it satisfies 
	\begin{align}
	u_{i}(t,x)&=S_{t}f_{i}(x)+\int_{0}^{t} S_{t-r}[\gamma_{i} u_{i}(r,x)+u_{j}^{1+\beta_{i}}(r,x)]dr+ k_{i1}\int_{0}^{t}S_{t-r}u_{i}(r,x)dW(r) \nonumber\\
	&\qquad+k_{i2}\int_{0}^{t}S_{t-r}u_{i}(r,x)dB^{H}(r),\ \mathbb{P}\mbox{-a.s.}, \nonumber
	\end{align}
	$i=1,2, \ j\in \left\{ 1,2 \right\} / \left\{i\right\},$ for each $t \in(0, \tau)$ and $x \in D$.
	\end{Def} 
	Here, the semigroup $\left\{ S_{t} \right\}_{t\geq 0}$  of bounded linear operators is defined by 
	\begin{eqnarray}
	S_{t}f(x)=\mathbb{E}\left[ f(X_{t}), \ t<\tau_{D}|X_{0}=x \right], \quad  x\in D,\nonumber
	\end{eqnarray}
	for all bounded and measurable $f : D\rightarrow \mathbb{R},$ where $\left\{ X_{t} \right\}_{t\geq 0}$ is the Brownian motion in $\mathbb{R}^{d}$ with variance parameter $2$, killed at the time $\tau_{D}$ at which it strikes the boundary $\partial D.$ As discussed above,  $\lambda>0$ is the first eigenvalue of $-\Delta$ on $D$, which satisfies 
	\begin{eqnarray} \label{a3}
	-\Delta \psi(x)=\lambda \psi(x), \ \ x\in D,
	\end{eqnarray}
	$\psi$ being the corresponding eigenfunction, which is strictly positive on $D$ and $\psi|_{\partial D} =0.$ Remember that $$S_{t}\psi=\exp\{{-\lambda t}\}\psi, \ \ t\geq 0,$$ where we  have assumed that $\psi$ is normalized so that $\int_{D} \psi(x)dx=1.$ Let $\left\lbrace p_{t}(x,y)\right\rbrace_{t>0}$ be the transition kernel of $\left\lbrace S_{t}\right\rbrace_{t \geq 0}$ and one can refer to \cite{wang1992} for its properties. From the semigroup theory \cite{pazy} and \cite[Theorem 2.1]{doz2023}, we note that for any bounded measurable initial data $f=(f_{1}, f_{2})^{\top}$, there exists a unique local mild solution $v=(v_1,v_2)^{\top}$ of the system of random PDEs \eqref{s1} (see below) by Theorem 2.5 in \cite{Atienza} which is the weak solution of the system \eqref{s1}. By using  a random transformation given in \eqref{ran1}, we obtain the existence and blow-up  of the weak solution $u=(u_1,u_2)^{\top}$ of the system \eqref{b1}-\eqref{b2}.

 	The remaining sections are organized as follows: The next section is devoted to obtain an associated system of random PDEs, by using  a random transformation (cf. \eqref{ran1}) to the system \eqref{b1}-\eqref{b2}, which is useful to establish  lower and upper bounds for the explosion time $\tau$. In Section \ref{sec3}, Theorem \ref{t2} establish a lower bound for the finite-time blow-up solution $u=(u_{1},u_{2})^{\top}$ of the system \eqref{b1}-\eqref{b2}. The special case of $(1+\beta_{1})k_{21}-k_{11}=(1+\beta_{2})k_{11}-k_{21} =:\rho_{1}, \
	(1+\beta_{1})k_{22}-k_{12}=(1+\beta_{2})k_{12}-k_{22}=:\rho_{2}$ is considered in Section \ref{sec4}. Under the above setting, the relevant exponential functionals of the form $\int_{0}^{t} \{\exp\{{\rho_1 W(r)+\rho_2 B^{H}(r)-ar}\}dr,$ and an upper bound for the finite-time blow-up ($\tau^*_{1}\ \& \ \tau_{2}^{\ast}$) is obtained explicitly in Theorem \ref{thm4.1}. By relaxing the above conditions,  the lower and upper bounds ($\tau_{**}, \ \tau^{**}_{1}$ and $\tau^{**}_{2}$) for the finite-time blow-up solution $u=(u_1,u_2)^{\top}$ to the system \eqref{b1}-\eqref{b2}  are established (Theorems \ref{thm5.1} and \ref{thm5.2}) and sufficient conditions for the global existence of a solution $u=(u_{1},u_{2})^{\top}$ (cf. Theorem  \ref{cor1}) to the system \eqref{b1}-\eqref{b2} when $\frac{1}{2}<H<1$ and the initital values  $f_{1}=C_{1}\psi$ and $f_{2}=C_{2}\psi$ for some positive constants $C_{1}$ and $C_{2}$ with $C_1 \leq C_2$. Also, sufficient conditions are provided with sharp bounds (cf. Theorem \ref{sh1}) in Section \ref{sec5}.  
	In Section \ref{s2}, we obtain an upper bound for the probability of finite-time blow-up	solution $u=(u_{1},u_{2})^{\top}$ of the system \eqref{b1}-\eqref{b2} and blow-up before a given fixed time $T>0$ with $\frac{1}{2}<H<1$ established in Theorem \ref{thm6}. Further, in Theorem \ref{thm6t} the upper bounds for the tail of $\tau_{1}^{\ast}$ and $\tau_{2}^{\ast}$ with the general dependent structure of the Brownian motion $W(\cdot)$ and fBm $B^{H}(\cdot)$. Next, we explicitly provide a lower bound for the probability of finite-time blow-up solution $u=(u_{1},u_{2})^{\top}$ of the system (\ref{b1})-\eqref{b2} for an appropriate choice of parameters, by using  the Malliavin calculus and the method adopted in \cite{doz2023} (the results established in Theorem \ref{thm6.5}). Also for $\frac{3}{4}<H<1$, when $W$ and $B^{H}$  are independent with $\rho_{1}=\rho_{2}=\rho$ (say), where $\rho_{1}$ and $\rho_{2}$ are given above, we obtain more explicit lower bound for the probability of finite-time blow-up solution $u=(u_{1},u_{2})^{\top}$ of the system \eqref{b1}-\eqref{b2} (cf. Theorem \ref{thm6.6}).
		
	\section{A System of Random PDEs}\label{sec2}
	In this section, we obtain a system of random PDEs by using the random transformations 
	\begin{align}\label{ran1}
	v_{i}(t,x) = \exp\{-k_{i1}W(t)-k_{i2}B^{H}(t)\}u_{i}(t,x), i=1,2,
	\end{align}
	for $t \geq 0,\ x \in D$ and $\frac{1}{2}<H<1$,  the system \eqref{b1}-\eqref{b2} can be transformed into a system of random PDEs:
	\begin{equation}\label{s1}
	\left\{
	\begin{aligned}
	\frac{\partial v_{i}(t,x)}{\partial t}&=\left( \Delta+\gamma_{i}-\frac{k_{i1}^{2}}{2} \right)v_{i}(t,x)+e^{-k_{i1}W(t)-k_{i2}B^{H}(t)}\left( e^{k_{j1}W(t)+k_{j2}B^{H}(t)}v_{j}(t,x) \right)^{1+\beta_{i}},\\
	v_{i}(0,x)&=f_{i}(x), \ \  x \in D, \\
	v_{i}(t,x)&=0, \ \ t \geq 0 \ \  x \in \partial D, 
	\end{aligned}
	\right.
	\end{equation}
	for $i=1,2, \ j\in \left\{ 1,2 \right\} / \left\{i\right\}.$ This system is understood in the pathwise sense and thus classical results for parabolic PDEs (see Friedman \cite[Chapter 7]{fried1964} and \cite{doz2023}) can be applied to show existence, uniqueness and positivity of solution $v=(v_1,v_2)^{\top}\in C([{0,\tau}];L^{\infty}(D;\mathbb{R}^2))$ up to an eventual blow-up. 
	
	The following result provides an equivalence between the weak solution $u=(u_1,u_2)^{\top}$ of the system of stochastic PDEs \eqref{b1}-\eqref{b2} and the system of random PDEs \eqref{s1}.
	\begin{theorem} \label{thm2.1}
	Let $u=(u_1,u_2)^{\top}$ be the weak solution of the system \eqref{b1}-\eqref{b2}. Then the function $v=(v_1,v_2)^{\top}$ defined by
	\begin{eqnarray}
	v_{i}(t,x) = \exp\{-k_{i1}W(t)-k_{i2}B^{H}(t)\}u_{i}(t,x), \ t \geq 0,\ x \in D,\ i=1,2, \nonumber 
	\end{eqnarray} 
	is a weak solution of the system of random PDEs \eqref{s1} and viceversa. 
	\begin{proof}
	By using Ito's formula \cite[Theorem 2.7.2]{mis2008}, we have for $\frac{1}{2}<H<1$
	\begin{align}
	e^{-k_{i1}W(t)-k_{i2}B^{H}(t)}=&1-\int_{0}^{t} e^{-k_{i1}W(s)-k_{i2}B^{H}(s)} (k_{i1}dW(s)+k_{i2}dB^{H}(s))\nonumber\\
	&+\frac{k_{i1}^{2}}{2}\int_{0}^{t} e^{-k_{i1}W(s)-k_{i2}B^{H}(s)}ds. \nonumber 
	\end{align}
	For any smooth function $\varphi_{i},\ i=1,2$  with compact support, we set
	$$u_{i}(t,\varphi_{i})=\int_{D}u_{i}(t,x) \varphi_{i}(x)dx, \ \ i=1,2.$$
	Then the weak solution $u=(u_1,u_2)^{\top}$ of the system \eqref{b1}-\eqref{b2} is given by
	\begin{align}\label{c1} 
	u_{i}\left(t,\varphi_{i}\right)&=u_{i}\left(0,\varphi_{i}\right)+\int_{0}^{t}u_{i}\left(s,\Delta \varphi_{i}\right)ds+\gamma_{i}\int_{0}^{t}u_{i}\left(s,\varphi_{i}\right)ds+\int_{0}^{t} u^{1+\beta_{i}}_{j}\left(s,\varphi_{i}\right)ds\nonumber\\
	&\quad+k_{i1}\int_{0}^{t}u_{i}\left(s,\varphi_{i}\right)dW(s)+k_{i2}\int_{0}^{t}u_{i}\left(s,\varphi_{i}\right)dB^{H}(s),
	\end{align}
	for $i=1,2,\ \left\lbrace j \right\rbrace=\left\lbrace1,2\right\rbrace/\{i\}.$ By applying the integration by parts formula, we obtain
	\begin{align*}
	v_{i}(t,\varphi_{i}):&=\int_{D}v_{i}(t,x)\varphi_{i}(x)dx \nonumber\\
    &=v_{i}(0,\varphi_{i})+\int_{0}^{t} e^{-k_{i1}W(t)-k_{i2}B^{H}(t)}du_{i}(s,\varphi_{i})+\int_{0}^{t} u_{i}(s,\varphi_{i})\Bigg(e^{-k_{i1}W(s)-k_{i2}B^{H}(s)} \Big( k_{i1}dW(s)\nonumber\\
    &\quad+k_{i2}dB^{H}(s)\Big) ds+\frac{k_{i1}^{2}}{2}\int_{0}^{t} e^{-k_{i1}W(s)-k_{i2}B^{H}(s)}ds\Bigg) -k_{i1}^{2}\int_{0}^{t}e^{-k_{i1}W(s)-k_{i2}B^{H}(s)} u_{i}(t, \varphi_{i})ds.
	\end{align*} 
	Therefore,
	\begin{align}\label{a6}
	v_{i}(t,\varphi_{i}) &= v_{i}(0,\varphi_{i})+\int_{0}^{t}v_{i}\left(s,\Delta \varphi_{i}\right)ds+\gamma_{i}\int_{0}^{t}v_{i}\left(s,\varphi_{i}\right)ds-\frac{k_{i1}^{2}}{2} \int_{0}^{t}v_{i}(s \varphi_{i})ds  \nonumber\\
	&\qquad +\int_{0}^{t}e^{-k_{i1}W(t)-k_{i2}B^{H}(t)} \left( e^{k_{j1}W(t)+k_{j2}B^{H}(t)} v_{j} \right)^{1+\beta_{i}} (s,\varphi_i)ds \nonumber\\
	&= v_{i}(0,\varphi_{i})+\int_{0}^{t}\left[ \Delta+\gamma_{i}-\frac{k_{i1}^{2}}{2}\right]  v_{i}\left(s,\varphi_{i}\right)ds \nonumber\\
	&\qquad+\int_{0}^{t}e^{-k_{i1}W(t)-k_{i2}B^{H}(t)} \left( e^{k_{j1}W(t)+k_{j2}B^{H}(t)} v_{j} \right)^{1+\beta_{i}} (s,\varphi_i)ds,  
	\end{align}
	for $i=1,2,\ \left\lbrace j \right\rbrace=\left\lbrace1,2\right\rbrace/\{i\}.$ The preceding equalities mean that $v=(v_1,v_2)^{\top}$ is a weak solution of the system  \eqref{s1}. The converse part follows from the fact that the change of variable is given by a homeomorphism  which transforms one random dynamical system into an another equivalent one.	
	\end{proof}
	\end{theorem}
   From the classical results available in semigroup theory \cite{pazy}, we note that for any bounded measurable initial data $f=(f_{1}, f_{2})^{\top},$ there exists a unique local mild solution $v=(v_1,v_2)^{\top}$ of the system of random PDEs \eqref{s1}. Precisely, there exists a number $0<\tau=\tau(\omega)\leq \infty$ such that $v=(v_1,v_2)^{\top}$ satisfies the integral equation on $(0, \tau)$
	\begin{align}\label{a2}
	v_{i}(t,x)&=\exp \left\{ \left( \gamma_{i}- \frac{k_{i1}^{2}}{2}\right) t \right\}S_{t}f_{i}(x) \\
	& \quad+\int_{0}^{t} e^{(t-r)\left( \gamma_{i}- \frac{k_{i1}^{2}}{2}\right) } S_{t-r}\left[ e^{-k_{i1}W(r)-k_{i2}B^{H}(r)}\left(e^{k_{j1}W(r)+k_{j2}B^{H}(r)}v_{j}(r,\cdot)\right)^{1+\beta_{i}} \right](x)dr, \nonumber
	\end{align}
	for $i=1,2, \ j\in \left\{ 1,2 \right\} / \left\{i\right\}$ and $x\in D$.  
	
	Let $\tau$ be the blow-up time of the system \eqref{s1} with the initial values of the above form. Due to Theorem \ref{thm2.1} and a.s continuity of $W(\cdot)$ and $B^{H}(\cdot)$, $\tau$ is also the blow-up time for the system  \eqref{b1}-\eqref{b2}.    Our aim is to find random times $\tau_{*}$ and $\tau^{*}$ such that $0< \tau_*\leq \tau\leq \tau^*$.

	\section{A Lower Bound for $\tau$}\label{sec3}
     This section aims to find a lower bound $\tau_*$ to the blow-up times such that $ \tau_*\leq \tau$ when $\frac{1}{2}<H<1$.	First, we consider the system \eqref{b1}-\eqref{b2} with the set of parameters $\gamma_{i}=\lambda+\frac{k_{i1}^{2}}{2},\ i=1,2,$ and 
	\begin{equation}\label{a4}
	\begin{aligned} 
	 (1+\beta_{1})k_{21}-k_{11}&=(1+\beta_{2})k_{11}-k_{21} =:\rho_{1},\\
	(1+\beta_{1})k_{22}-k_{12}&=(1+\beta_{2})k_{12}-k_{22}=:\rho_{2}. 
	\end{aligned}
	\end{equation}
	The following result provides a lower bound for the finite-time
	blow-up solution $u=(u_{1},u_{2})^{\top}$ of the system \eqref{b1}-\eqref{b2}.
	\begin{theorem}\label{t2} Assume that the conditions given in  (\ref{a4}) holds, and let the initial values be of the form
	\begin{eqnarray} \label{q1}
	f_{1}=C_{1} \psi \ \ \mbox{and} \ \ f_{2}=C_{2} \psi ,
	\end{eqnarray} 
	for some positive constants $C_{1}$ and $C_{2}$ with $C_{1}\leq C_{2}$. Let $\tau_{\ast}$ is given by
	\begin{align}
	\tau_{\ast} = \inf \Bigg\{ t\geq 0 : &\int_{0}^{t} \exp\{\rho_{1} W(r)+\rho_{2}B^{H}(r)\}dr 
	\geq \frac{1}{\beta_{1}C^{\beta_{1}}_{1}\left\|\psi \right\|_{\infty}^{\beta_{1}}}, \nonumber\\
	(or) &\int_{0}^{t} \exp\{\rho_{1} W(r)+\rho_{2}B^{H}(r)\}dr 
	\geq \frac{1}{\beta_{2}C^{\beta_{2}}_{2}\left\|\psi \right\|_{\infty}^{\beta_{2}}}	\Bigg\}. \nonumber
	\end{align}
	Then $\tau_{\ast}\leq\tau$, where $\|\psi\|_{\infty}=\displaystyle \sup_{x \in D}\psi(x).$
	\begin{proof}
	Let $v_{1}(\cdot,\cdot)$ and $v_{2}(\cdot,\cdot)$ solve \eqref{a2}. Then, we have
	\begin{align}
	v_{1}(t,x)&=\exp\{{\lambda t}\}S_{t}f_{1}(x)+\int_{0}^{t}\exp\{{\lambda (t-r)}\}S_{t-r} \left( \exp\{{\rho_{1} W(r)+\rho_{2}B^{H}(r)}\}v^{1+\beta_{1}}_{2}(r,x)\right)dr,\nonumber\\
	v_{2}(t,x)&=\exp\{{\lambda t}\}S_{t}f_{2}(x)+\int_{0}^{t}\exp\{{\lambda (t-r)}\}S_{t-r} \left( \exp\{{\rho_{1} W(r)+\rho_{2}B^{H}(r)}\}v^{1+\beta_{2}}_{1}(r,x)\right)dr, \nonumber
	\end{align}
	for all $x\in D,  \ t\geq 0.$ Let us define the operators $\mathcal{T}_{1}$ and $\mathcal{T}_{2}$ as
	\begin{align}
	\mathcal{T}_{1} v(t,x)&=\exp\{{\lambda t}\}S_{t}f_{1}(x)+\int_{0}^{t}\exp\{{\rho_{1} W(r)+\rho_{2} B^{H}(r)+\lambda (t-r)}\}\left( S_{t-r}v \right)^{1+\beta_{1}}dr,\nonumber\\
	\mathcal{T}_{2} v(t,x)&=\exp\{{\lambda t}\}S_{t}f_{2}(x)+\int_{0}^{t}\exp\{{\rho_{1} W(r)+\rho_{2} B^{H}(r)+\lambda (t-r)}\}\left( S_{t-r}v \right)^{1+\beta_{2}}dr,\nonumber
	\end{align}
	where $v(\cdot,\cdot)$ is any non-negative, bounded and measurable function.	
		
	First we shall prove that \begin{eqnarray}
	v_{1}(t,x)=\mathcal{T}_{1} v_{2}(t,x), \  v_{2}(t,x)=\mathcal{T}_{2} v_{1}(t,x), \ \ x\in D, \ \ 0 \leq t<\tau_{*},\nonumber
	\end{eqnarray}
	for some non-negative, bounded and measurable functions $v_1(\cdot,\cdot)$ and $v_2(\cdot,\cdot)$. Moreover, on the set $t\leq \tau_{*},$ we set
	\begin{align}
	\mathscr{G}_{1}(t)&=\left[ 1-\beta_{1} \int_{0}^{t}  \exp\{\rho_{1} W(r)+\rho_{2} B^{H}(r)\}\left\| \exp\{{ \lambda r}\}S_{r}f_{1} \right\|_{\infty}^{\beta_{1}}dr\right]^{\frac{-1}{\beta_{1}}},\nonumber\\
	\mathscr{G}_{2}(t)&=\left[ 1-\beta_{2} \int_{0}^{t}  \exp\{\rho_{1} W(r)+\rho_{2} B^{H}(r)\}\left\| \exp\{{\lambda r}\}S_{r}f_{2} \right\|_{\infty}^{\beta_{2}}dr\right]^{\frac{-1}{\beta_{2}}}.\nonumber
	\end{align}
	Then, it can be easily seen that 
	\begin{eqnarray}
	\frac{d \mathscr{G}_{1}(t)}{dt}=\exp\{\rho_{1} W(t)+\rho_{2} B^{H}(t)\}\left\| \exp\{{\lambda t}\}S_{t}f_{1} \right\|_{\infty}^{\beta_{1}} \mathscr{G}^{1+\beta_{1}}_{1}(t), \ \ \mathscr{G}_{1}(0)=1, \nonumber
	\end{eqnarray}
	so that 
	\begin{eqnarray}
	\mathscr{G}_{1}(t)=1+\int_{0}^{t} \exp\{\rho_{1} W(r)+\rho_{2} B^{H}(r)\} \left\| \exp\{{\lambda r}\}S_{r}f_{1} \right\|_{\infty}^{\beta_{1}} \mathscr{G}^{1+\beta_{1}}_{1}(r) dr. \nonumber
	\end{eqnarray}
	Similarly, we have 
	\begin{eqnarray}
	\mathscr{G}_{2}(t)=1+\int_{0}^{t} \exp\{\rho_{1} W(r)+\rho_{2} B^{H}(r)\}  \left\| \exp\{{\lambda r}\}S_{r}f_{2} \right\|_{\infty}^{\beta_{2}} \mathscr{G}^{1+\beta_{2}}_{2}(r) dr. \nonumber	
	\end{eqnarray}
	Let us choose $v(\cdot,\cdot)$ such that $$0 \leq v(t,x)\leq \exp\{{\lambda t}\}S_{t}f_{1}(x)\mathscr{G}_{1}(t),$$ for $x\in D$ and $t<\tau_{*}.$ Then $\exp\{{\lambda t}\}S_{t}f_{1}(x)\leq \mathcal{T}_{1} v(t,x)$ and
	\begin{align}
	\mathcal{T}_{1}v(t,x)&=\exp\{{\lambda t}\}S_{t}f_{1}(x)+\int_{0}^{t}\exp\{{\rho_{1} W(r)+\rho_{2} B^{H}(r)+\lambda (t-r)}\}\left( S_{t-r}v \right)^{1+\beta_{1}}dr\nonumber\\
	&\leq \exp\{{\lambda t}\}S_{t}f_{1}(x)+\int_{0}^{t}\exp\{{\rho_{1} W(r)+\rho_{2} B^{H}(r)+\lambda (t-r)}\}\nonumber\\ 
	& \qquad\times\left[ \exp\{{\lambda r}\} \mathscr{G}_{1}(r)\exp\{{\beta_{1}\lambda r}\} \mathscr{G}^{\beta_{1}}_{1}(r) \left\|S_{r}f_{1}\right\|_{\infty}^{\beta_{1}}S_{t-r}(S_{r}f_{1}(x)) \right]dr \nonumber\\
	&=\exp\{{\lambda t}\}S_{t}f_{1}(x)+\int_{0}^{t}\exp\{{\rho_{1} W(r)+\rho_{2} B^{H}(r)+\lambda t}\}\nonumber\\ 
	& \qquad\times\left[\exp\{{\beta_{1}\lambda r}\} \mathscr{G}^{1+\beta_{1}}_{1}(r) \left\|S_{r}f_{1}\right\|_{\infty}^{\beta_{1}}S_{t}f_{1}(x) \right]dr \nonumber\\
	&= \exp\{{\lambda t}\}S_{t}f_{1}(x)\Bigg\{ 1+ \int_{0}^{t} \exp\{{\rho_{1} W(r)+\rho_{2} B^{H}(r)}\}\left\| \exp\{{\lambda r}\}S_{r}f_{1} \right\|_{\infty}^{\beta_{1}} \mathscr{G}^{1+\beta_{1}}_{1}(r)dr\Bigg\} \nonumber\\
	&=\exp\{{\lambda t}\} S_{t}f_{1}(x)\mathscr{G}_{1}(t).\nonumber
	\end{align}
	Thus, we have
	$$\exp\{{\lambda t}\}S_{t}f_{1}(x)\leq \mathcal{T}_{1} v(t,x)  \leq \exp\{{\lambda t}\} S_{t}f_{1}(x)\mathscr{G}_{1}(t).$$
	Similarly, we get 
	\begin{eqnarray}
	\exp\{{\lambda t}\}S_{t}f_{2}(x)\leq \mathcal{T}_{2} w(t,x)\leq \exp\{{\lambda t}\}S_{t}f_{2}(x) \mathscr{G}_{2}(t), \nonumber
	\end{eqnarray}
	for all $w(\cdot,\cdot)$ such that $$0\leq w(t,x)\leq \exp\{{\lambda t}\} S_{t}f_{2}(x)\mathscr{G}_{2}(t).$$ Let us take
	\begin{eqnarray}\label{in1}
	u^{(0)}_{1}(t,x)=\exp\{{\lambda t}\}S_{t}f_{1}(x), \ \ u^{(0)}_{2}(t,x)=\frac{C_{1}}{C_{2}}\exp\{{\lambda t}\}S_{t}f_{2}(x), 
	\end{eqnarray}  
	and
	\begin{eqnarray}\label{in2}
	u^{(n)}_{1}(t,x)=\mathcal{T}_{1} u^{(n-1)}_{2}(t,x), \ \ u^{(n)}_{2}(t,x)=\mathcal{T}_{2} u^{(n-1)}_{1}(t,x), \ \ n\geq1, 
	\end{eqnarray}
	for $x\in D,\ 0\leq t<\tau_{*}.$ Our aim is to show that the sequences of functions $\{ u^{(n)}_{1} \} \ \mbox{and} \ \{u^{(n)}_{2}\}$ are increasing.
	We  consider
	\begin{align}
	u^{(0)}_{1}(t,x)&\leq \exp\{{\lambda t}\} S_{t}f_{1}(x)+\int_{0}^{t} \exp\{{\rho_{1} W(r)+\rho_{2} B^{H}(r)+\lambda (t-r)}\} \left( S_{t-r}u^{(0)}_{2}(r,x) \right)^{1+\beta_{1}}dr\nonumber\\
	&= \mathcal{T}_{1}u^{(0)}_{2}(t,x) = u^{(1)}_{1}(t,x).\nonumber
	\end{align}
	Assuming that $u^{(n)}_{i}\geq u^{(n-1)}_{i}, \ i=1,2,$ for some $n \geq 1,$  the monotonicity of $\mathcal{T}_{1}$ leads to the inequality
	\begin{eqnarray}
	u^{(n+1)}_{1}=\mathcal{T}_{1}  u^{(n)}_{2}(t,x)\geq \mathcal{T}_{1} u^{(n-1)}_{2}(t,x)=u^{(n)}_{1},\nonumber
	\end{eqnarray}
	and the monotonicity of $\mathcal{T}_{2}$ results in
	\begin{eqnarray}
	u^{(n+1)}_{2}=\mathcal{T}_{2}  u^{(n)}_{1}(t,x)\geq \mathcal{T}_{2} u^{(n-1)}_{1}(t,x)=u^{(n)}_{2},\nonumber
	\end{eqnarray}
	and therefore the limits 
	\begin{eqnarray}
	v_{1}(t,x)=\lim_{n\rightarrow \infty}u^{(n)}_{1}(t,x), \ \ v_{2}(t,x)=\lim_{n\rightarrow \infty}u^{(n)}_{2}(t,x), \nonumber
	\end{eqnarray}
	exist for $x\in D$ and $0 \leq t< \tau_{*}.$ Then by the monotone convergence theorem, we obtain 
	\begin{eqnarray}
	v_{1}(t,x)=\mathcal{T}_{1} v_{2}(t,x), \  v_{2}(t,x)=\mathcal{T}_{2} v_{1}(t,x), \ \ x\in D, \ \ 0 \leq t<\tau_{*}.\nonumber
	\end{eqnarray}
	Since $f_{1}=C_{1} \psi$ and $f_{2}=C_{2} \psi,$ for some positive constants $C_{1}$ and $C_{2}$ with $C_{1} \leq C_{2}$ following similarly as in the paper \cite[Theorem 3.1]{skm1}, we obtain
		\begin{align}
		\exp\{\lambda t\}S_{t} f_{1}(x)&\leq u^{(n)}_{1}(t,x) \leq \exp\{\lambda t\}S_{t} f_{1}(x)\mathscr{G}_{1}(t), \nonumber\\
		\exp\{\lambda t\}S_{t} f_{2}(x)&\leq u^{(n)}_{2}(t,x) \leq \exp\{\lambda t\}S_{t} f_{2}(x)\mathscr{G}_{2}(t). \nonumber
		\end{align}
	Moreover, we have
	\begin{align} 
	v_{1}(t,x)&=\mathcal{T}_{1} v_{2} (t,x)\leq \exp\{{\lambda t}\}S_{t}f_{1}(x) \mathscr{G}_{1}(t)\ \text{ and } \nonumber \\ v_{2}(t,x)&=\mathcal{T}_{2} v_{1}(t,x)\leq \exp\{{\lambda t}\}S_{t}f_{2}(x) \mathscr{G}_{2}(t), \nonumber
	\end{align}
	so that
	\begin{align}
	v_{1}(t,x) &\leq \frac{\exp\{{\lambda t}\}S_{t}f_{1}(x)}{\left[ 1-\beta_{1}\int_{0}^{t}  \exp\{\rho_{1} W(r)+\rho_{2} B^{H}(r)\}\left\| \exp\{{\lambda r}\}S_{r}f_{1} \right\|_{\infty}^{\beta_{1}}dr\right]^{\frac{1}{\beta_{1}}}}, \nonumber\\
	v_{2}(t,x) &\leq \frac{\exp\{{\lambda t}\}S_{t}f_{2}(x)}{\left[ 1-\beta_{2}\int_{0}^{t}  \exp\{\rho_{1} W(r)+\rho_{2} B^{H}(r)\}\left\| \exp\{{\lambda r}\}S_{r}f_{2} \right\|_{\infty}^{\beta_{2}}dr\right]^{\frac{1}{\beta_{2}}}}. \nonumber	
	\end{align}
	By the choice of initial values as in (\ref{q1}), the proof of the theorem follows.
	\end{proof}
	\end{theorem}
	\begin{Rem} 
	For  general bounded, measurable and positive $f=(f_{1}, f_{2})^{\top}$ the blow-up time of (\ref{b1}) is bounded below by the random time 
	\begin{align}
	\inf \Bigg\{ &t\geq 0 : \int_{0}^{t} \exp\{{\rho_{1}W(r)+\rho_{2}B^{H}(r)}\} \left\| \exp\{{\lambda r}\}S_{r}f_{1} \right\|_{\infty}^{\beta_{1}}dr \geq \frac{1}{\beta_{1}} \nonumber\\
	&\mbox{ or }\int_{0}^{t} \exp\{{\rho_{1}W(r)+\rho_{2}B^{H}(r)}\} \left\| \exp\{{\lambda r}\}S_{r}f_{2} \right\|_{\infty}^{\beta_{2}}dr \geq  \frac{1}{\beta_{2}}\Bigg\}, \nonumber 
	\end{align} 
	which coincides with the lower bound $\tau_{\ast}$ in Theorem \ref{t2}, when the initial values satisfy (\ref{q1}).
	\end{Rem}
	
	\section{Upper Bound for $\tau$}\label{sec4}
	In this section, we obtain an upper bound $\tau^*$ for the blow-up time $\tau$, under the assumption \eqref{a4} and for any initial values $f=(f_{1}, f_{2})^{\top}$ with $\frac{1}{2}<H<1$. By setting $\varphi_{i}=\psi, \ i=1,2$, the system (\ref{a6}) reduces to
	\begin{align} \label{a5}
	v_{i}(t,\psi) &= v_{i}(0,\psi) + \int_{0}^{t} \left( \Delta +\gamma_{i}-\frac{k_{i1}^{2}}{2}\right) v_{i}(s, \psi)ds  \nonumber\\
	&\quad+ \int_{0}^{t} \exp\{{-k_{i1} W(s)-k_{i2} B^{H}(s)}\}  (\exp\{{k_{j1} W(s)+k_{j2} B^{H}(s)}\} v_{j})^{1+\beta_{i}}(s,\psi) ds. 
	\end{align}	
    Using the fact that $\psi(x)=0$ for $x \in \partial D,$ we have
    \begin{equation}
    \begin{aligned}
    v_{i}(s,\Delta \psi)=\int_{D} v_{i}(s,x)\Delta \psi(x)dx&=\int_{D} \Delta v_{i}(s,x)\psi(x)dx=\Delta v_{i}(s,\psi)=-\lambda v_{i}(s,\psi).\nonumber 
    \end{aligned}
    \end{equation}
    Therefore, (\ref{a5}) becomes
	\begin{align*}
	\frac{dv_{i}(t,\psi)}{dt} &=\left( -\lambda+\gamma_{i}-\frac{k_{i1}^{2}}{2}\right) v_{i}(t,\psi)+e^{-k_{i1} W(s)-k_{i2} B^{H}(s)}\left(e^{k_{j1} W(s)+k_{j2} B^{H}(s)} v_{j}\right)^{1+\beta_{i}}(t,\psi), \nonumber
	\end{align*}
	for $i=1,2, \ j\in \left\{ 1,2 \right\} / \left\{i\right\}.$ By using Jensen's inequality, we obtain 
	\begin{align}\label{j1}
	\lefteqn{\hspace{-0.8in}(\exp\{{k_{j1} W(s)+k_{j2} B^{H}(s)}\} v_{j})^{1+\beta_{i}}(s,\psi)}\nonumber\\
	&\geq \left[ \int_{D} (\exp\{{k_{j1} W(s)+k_{j2} B^{H}(s)}\}v_{j}(s,x))\psi(x)dx\right]^{1+\beta_{i}}\nonumber\\ 
	&=\exp\{{(1+\beta_{i})k_{j1} W(s)+(1+\beta_{i})k_{j2} B^{H}(s)}\}v_{j}(s,\psi)^{1+\beta_{i}}.
	\end{align}
	Thus, it is immediate that 
	\begin{align}
	\frac{dv_{i}(t,\psi)}{dt} &\geq \left( -\lambda+\gamma_{i}-\frac{k_{i1}^{2}}{2}\right) v_{i}(t,\psi)+\exp\{{-k_{i1} W(s)-k_{i2} B^{H}(s)}\}\nonumber\\
	&\quad\times \exp\{{(1+\beta_{i})k_{j1} W(s)+(1+\beta_{i})k_{j2} B^{H}(s)}\}v_{j}(s,\psi)^{1+\beta_{i}}, \nonumber
	\end{align}
	$i=1,2, \ j\in \left\{ 1,2 \right\} / \left\{i\right\}.$	
	In this way, $v_{i}(t,\psi) \geq h_{i}(t)\geq 0,$ for $ i=1,2,$ where
	\begin{equation} \label{ab2}
	\left\{
	\begin{aligned}
	\frac{dh_{1}(t)}{dt}&=\left( -\lambda+\gamma_{1}-\frac{k_{11}^{2}}{2}\right) h_{1}(t)+\exp\{{\rho_{1} W(t)+\rho_{2} B^{H}(t)}\} h_{2}^{1+\beta_{1}}(t)  \\ \frac{dh_{2}(t)}{dt}&=\left( -\lambda+\gamma_{2}-\frac{k_{21}^{2}}{2}\right)h_{2}(t)+\exp\{{\rho_{1} W(t)+\rho_{2} B^{H}(t)}\} h_{1}^{1+\beta_{2}}(t),\\
	h_{i}(0)&=v_{i}(0,\psi), \ i=1,2. 
	\end{aligned}
	\right.
	\end{equation}
	Let  us define $E(t):=h_{1}(t)+h_{2}(t)\geq 0, \ t\geq 0,$ so that $E(\cdot)$ satisfies 
	\begin{eqnarray}\label{e1}
	\frac{dE(t)}{dt}\geq (-\lambda+\gamma-k^{2})E(t)+\exp\{{\rho_{1} W(t)+\rho_{2} B^{H}(t)}\} \left[ h_{1}^{1+\beta_{2}}(t)+h_{2}^{1+\beta_{1}}(t) \right],
	\end{eqnarray}
	where $\gamma=\min\{\gamma_1,\gamma_2\}$ and $k^{2}=\max\left\lbrace  \frac{k_{11}^{2}}{2}, \frac{k_{21}^{2}}{2} \right\rbrace$.
	
	The following theorem gives an upper bound for the finite-time blow-up of solution $u =(u_{1}, u_{2})^{\top}$ of the system \eqref{b1}-\eqref{b2}.	
	\begin{theorem} \label{thm4.1} 
	Assume that $\beta_{1} \geq \beta_{2}>0$, for any initial values $f=(f_{1}, f_{2})^{\top}$, we have the following results:
	\begin{itemize}
	\item [1.] If  $\beta_{1}=\beta_{2}=\beta\ (say)>0$, let $\tau^{\ast}_{1}$ is given by
	\begin{eqnarray}\label{ST1}
	\tau^{\ast}_{1} = \inf \left\lbrace t\geq 0 : \int_{0}^{t} \exp\{{\textcolor{red}{\rho_{1}} W(s)+\textcolor{red}{\rho_{2}} B^{H}(s)+\beta(-\lambda+\gamma-k^{2})s}\} ds\geq 2^{\beta} {\beta}^{-1}E^{-\beta}(0) \right\rbrace. 
	\end{eqnarray}  
	Then $\tau \leq \tau^{\ast}_{1}$, where $\rho_{1}$ and $\rho_{2}$ are defined in \eqref{a4}. \\
	\item [2.] If $\beta_{1}>\beta_{2}>0$ and let $D_{1}=\left( \frac{\beta_{1}-\beta_{2}}{1+\beta_{1}} \right)\left(\frac{1+\beta_{1}}{1+\beta_{2}} \right)^{\frac{1+\beta_{2}}{\beta_{1}-\beta_{2}}},$
	\begin{align}\label{nca3}
	\epsilon_{0}&\leq \min \Bigg\{ 1, \left( h_{\textcolor{red}{2}}(0)/D_{1}^{1/1+\beta_{2}} \right)^{\beta_{1}-\beta_{2}} \Bigg\}.
	\end{align} 
	Assume that 
	\begin{align}\label{nca4}
	2^{-(1+\beta_{2})}\epsilon_{0}E^{1+\beta_{2}}(0)\geq \epsilon_{0}^{\frac{1+\beta_{1}}{\beta_{1}-\beta_{2}}}D_{1}.
	\end{align}
	Then $\tau\leq\tau^{\ast}_{2}$, where $\tau^{\ast}_{2}$ is given by
	\begin{align}\label{ST2}
	\tau^{\ast}_{2}& = \inf \left\lbrace  t\geq 0 : \int_{0}^{t} \exp\{\rho_{1} W(s)+\rho_{2} B^{H}(s)+\beta_{2}(-\lambda+\gamma-k^{2})s\}ds \right.\nonumber\\&\qquad\qquad\left. \geq {\left[\beta_{2} E^{\beta_{2}}(0)\left(\frac{\epsilon_0}{2^{1+\beta_{2}}}-\frac{ \epsilon_{0}^{\frac{1+\beta_{1}}{\beta_{1}-\beta_{2}}}D_{1}}{E^{1+\beta_{2}}(0)}\right)\right]^{-1}}\right\rbrace, 
	\end{align}
	where $\rho_{1}$ and $\rho_{2}$ are defined in \eqref{a4}, $\gamma=\min\{\gamma_1,\gamma_2\}$, $E(0)=\displaystyle\int_{D}[f_{1}(x)+f_{2}(x)]\psi(x)dx$ and $k^{2}=\max\left\lbrace  \frac{k_{11}^{2}}{2}, \frac{k_{21}^{2}}{2} \right\rbrace$.
    \end{itemize}
    \end{theorem}	
 
    \begin{Rem}
    Note that \eqref{nca4}  follows from the condition $\int_{D} [f_{1}(x)+f_2(x)] \psi (x)dx > 2 \epsilon_{0}^{\frac{1}{\beta_{1}-\beta_{2}}} D_{1}^{1/1+\beta_{2}}.$
    \end{Rem}
   
  \begin{proof}[Proof of Theorem \ref{thm4.1}] 
  	\textbf{Case 1:} Suppose that $\beta_{1}=\beta_{2}=\beta>0$ in (\ref{a4}) and \eqref{e1}. Then, we get
  	\begin{eqnarray} \label{a7}
  	\frac{dE(t)}{dt}\geq (-\lambda+\gamma-k^{2})E(t)+2\exp\{{\textcolor{red}{\rho_{1}} W(t)+\textcolor{red}{\rho_{2}} B^{H}(t)}\} \left[ h_{1}^{1+\beta}(t)+h_{2}^{1+\beta}(t) \right].
  	\end{eqnarray}
  	By substituting $a=1, \ b=\frac{h_{1}}{h_{2}},\ n=1+\beta$ into the inequality 
  	\begin{eqnarray}\label{cta1}
  	a^{n}+b^{n}\geq 2^{-n}{(a+b)}^{n},
  	\end{eqnarray}
  	which is valid for $a, \ b\geq0,$ we obtain 
  	\begin{eqnarray}
  	&&h_{1}^{1+\beta}(t)+h_{2}^{1+\beta}(t) \geq 2^{-(1+\beta)} \left( h_{1}(t)+h_{2}(t) \right)^{1+\beta} \geq 2^{-(1+\beta)} E^{1+\beta}(t).\nonumber
  	\end{eqnarray}
  	From $\left( \ref{a7} \right) $, we infer that 
  	\begin{align*}
  	\frac{dE(t)}{dt}\geq(-\lambda+\gamma-k^{2})E(t)+ 2^{-\beta} \exp\{{\textcolor{red}{\rho_{1}} W(t)+\textcolor{red}{\rho_{2}} B^{H}(t)}\} E^{1+\beta}(t).\nonumber
  	\end{align*} 
  	Thus, by comparison argument (see \cite[Theorem 1.3]{teschl} and \cite[Appendix]{smk2}), we have $E(t)$ blows up not later than the solution $I(t)$ of the differential equation
  	\begin{equation}
  		\left\{
  	\begin{aligned}
  	\frac{dI(t)}{dt}&=(-\lambda+\gamma-k^{2})I(t)+ 2^{-\beta} \exp\{{\textcolor{red}{\rho_{1}} W(t)+\textcolor{red}{\rho_{2}} B^{H}(t)}\} I(t)^{1+\beta}, \\  I(0)&=E(0),\nonumber
  	\end{aligned}
  \right.
\end{equation}
  	which immediately  gives 
  	\begin{align}
  	I(t)=e^{(\lambda-\gamma+k^{2})t}\left\lbrace E(0)^{-\beta}- 2^{-\beta}\beta\int_{0}^{t} \exp\{{\textcolor{red}{\rho_{1}} W(t)+\textcolor{red}{\rho_{2}} B^{H}(t)+\beta(-\lambda+\gamma-k^{2})s}\} ds\right\rbrace^{-\frac{1}{\beta}}. \nonumber
  	\end{align} 
  	From  the above equation, the explosion time is given by
  	\begin{align*}
  	\tau^{\ast}_{1} = \inf \left\lbrace t\geq 0 : \int_{0}^{t} \exp\{{\textcolor{red}{\rho_{1}} W(t)+\textcolor{red}{\rho_{2}} B^{H}(t)+\beta(-\lambda+\gamma-k^{2})s}\} ds\geq 2^{\beta} {\beta}^{-1}E^{-\beta}(0) \right\rbrace. 
  	\end{align*}
   
    \vskip 0.2 cm
    \noindent 
    \textbf{Case 2:} Suppose $\beta_{1}>\beta_{2}>0$ in  \eqref{e1},  then we have 
    \begin{align} \label{nca1}
    \frac{dE(t)}{dt}&\geq(-\lambda+\gamma-k^{2})E(t)+  \exp\{\rho_{1} W(t)+\rho_{2} B^{H}(t)\}  \left[ h_{1}^{1+\beta_{2}}(t)+h_{2}^{1+\beta_{1}}(t) \right].
   \end{align}
    The Young inequality states that if $1<b<\infty$, $\delta>0$, and $a=\frac{b}{b-1}$, then 
    \begin{eqnarray} \label{na10}
    xy\leq \frac{\delta^{a}x^{a}}{a}+\frac{\delta^{-b}y^{b}}{b}, \ x,y\geq 0. 
    \end{eqnarray}
    By setting $b=\frac{1+\beta_{1}}{1+\beta_{2}}, \ y=h_{2}^{1+\beta_{2}}(t), \ x=\epsilon, \ \delta=\left( \frac{1+\beta_{1}}{1+\beta_{2}} \right)^{\frac{1+\beta_{2}}{1+\beta_{1}}},$ and using the fact that $\beta_{2}<\beta_{1}$ in (\ref{na10}), it follows that for any $\epsilon>0,$
    \begin{eqnarray} 
    h_{2}^{1+\beta_{1}}(t) \geq \epsilon h_{2}^{1+\beta_{2}}(t)-D_{1}\epsilon^{\frac{1+\beta_{1}}{\beta_{1}-\beta_{2}}}\nonumber.
    \end{eqnarray}
    Since $\epsilon_0$ is the minimum of the quantities given in  \eqref{nca3},  in particular, we have $\epsilon_0\leq ( h_{\textcolor{red}{2}}(0)/D_{1}^{1/1+\beta_{2}})^{\beta_{1}-\beta_{2}},$ so that 
    \begin{eqnarray}
    \epsilon_{0} h_{2}^{1+\beta_{2}}(0)-D_{1}\epsilon_{0}^{\frac{1+\beta_{1}}{\beta_{1}-\beta_{2}}}\geq 0\nonumber.
    \end{eqnarray}
    From  \eqref{nca1},  we have
    \begin{align}\label{nca2}
    \frac{dE(t)}{dt}\geq(-\lambda+\gamma-k^{2})E(t)+ \exp\{\rho_{1} W(t)+\rho_{2} B^{H}(t)\} \left[ h_{1}^{1+\beta_{2}}(t)+\epsilon_{0} h_{2}^{1+\beta_{2}}(t)-D_{1}\epsilon_{0}^{\frac{1+\beta_{1}}{\beta_{1}-\beta_{2}}} \right].
    \end{align}
   	Utilizing the inequality  \eqref{cta1}  again with $n=1+\beta_{2}, \ a=1$ and $b=\epsilon_{0}^\frac{1}{1+\beta_{2}}\frac{h_{1}(t)}{h_{2}(t)}$, and by using  \eqref{nca3},  we see that 
   	\begin{eqnarray}
   	h_{\textcolor{red}{1}}^{1+\beta_{2}}(t) + \epsilon_{0} h_{\textcolor{red}{2}}^{1+\beta_{2}}(t) \geq 2^{-(1+\beta_{2})} \left( h_{\textcolor{red}{1}}(t)+\epsilon_{0}^\frac{1}{1+\beta_{2}}h_{\textcolor{red}{2}}(t) \right)^{1+\beta_{2}}\geq  2^{-(1+\beta_{2})}\epsilon_{0}E^{1+\beta_{2}}(t)\nonumber.\end{eqnarray}
   	From  \eqref{nca2},  we deduce that
   \begin{align}
   \frac{dE(t)}{dt}\geq(-\lambda+\gamma-k^{2})E(t)+ \exp\{\rho_{1} W(t)+\rho_{2} B^{H}(t)\} \left[2^{-(1+\beta_{2})}\epsilon_{0}E^{1+\beta_{2}}(t)-D_{1}\epsilon_{0}^{\frac{1+\beta_{1}}{\beta_{1}-\beta_{2}}} \right].\nonumber
   \end{align}
   Note that the condition  \eqref{nca4}  gives $E(t)\geq E(0)>0,$ and so 
   \begin{align}
   \frac{dE(t)}{E^{1+\beta_{2}}(t)}\geq\frac{(-\lambda+\gamma-k^{2})}{E^{\beta_{2}}(t)}dt+\exp\{\rho_{1} W(t)+\rho_{2} B^{H}(t)\} \left[\frac{\epsilon_{0}}{2^{1+\beta_{2}}}-\frac{D_{1}\epsilon_{0}^{\frac{1+\beta_{1}}{\beta_{1}-\beta_{2}}}}{E^{1+\beta_{2}}(0)} \right]dt.\nonumber
   \end{align}
   Thus, by comparison argument (see \cite[Theorem 1.3]{teschl} and \cite[Appendix]{smk2}) $E(t) \geq I(t),$ for all $t \geq 0,$ where $I(t)$ solves the equation
   \begin{equation}
   	\left\{
   \begin{aligned}
   \frac{dI(t)}{I^{1+\beta_{2}}(t)}& = \frac{(-\lambda+\gamma-k^{2})}{I^{\beta_{2}}(t)}dt+\exp\{\rho_{1} W(t)+\rho_{2} B^{H}(t)\} \left[\frac{\epsilon_{0}}{2^{1+\beta_{2}}}-\frac{D_{1}\epsilon_{0}^{\frac{1+\beta_{1}}{\beta_{1}-\beta_{2}}}}{I^{1+\beta_{2}}(0)} \right]dt, \\ I(0)&=E(0),\nonumber
   \end{aligned}
\right.
\end{equation}
   and
   \begin{align*}
   I(t)= e^{(\lambda-\gamma+k^{2})t}\left\lbrace E^{-\beta_{2}}(0)-\beta_{2}\left[ \frac{\epsilon_0}{2^{1+\beta_{2}}}-\frac{ \epsilon_{0}^{\frac{1+\beta_{1}}{\beta_{1}-\beta_{2}}}D_{1} }{E^{1+\beta_{2}}(0)} \right] \displaystyle\int_{0}^{t} e^{\rho_{1} W(s)+\rho_{2} B^{H}(s)+\beta_{2}(-\lambda+\gamma-k^{2})s}ds \right\rbrace^{\frac{1}{-\beta_{2}}}.
   \end{align*}
   From the above equality, the explosion time is given by
   \begin{eqnarray}
   \tau^{\ast}_{2} = \inf \left\lbrace  t\geq 0 : \int_{0}^{t} e^{\rho_{1} W(s)+\rho_{2} B^{H}(s)+\beta_{2}(-\lambda+\gamma-k^{2})s}ds \geq {\left[\beta_{2} E^{\beta_{2}}(0)\left(\frac{\epsilon_0}{2^{1+\beta_{2}}}-\frac{ \epsilon_{0}^{\frac{1+\beta_{1}}{\beta_{1}-\beta_{2}}}D_{1}}{E^{1+\beta_{2}}(0)}\right)\right]^{-1}}\right\rbrace, \nonumber
   \end{eqnarray}
   which completes the proof. 
   \end{proof}

   \section{A More General Case}\label{sec5}
   In this section, we consider the system \eqref{b1}-\eqref{b2} with  the assumption that $\beta_{1}\geq \beta_{2}>0$, for any initial values $f=(f_{1}, f_{2})^{\top}$ and $\frac{1}{2}<H<1$. From (\ref{a2}), we have 
   \begin{align*}
   v_{i}(t,x)&=\exp\left\lbrace {\left( \gamma_i-\frac{k_{i1}^{2}}{2}\right)  t}\right\rbrace  S_{t}f_{i}(x)\nonumber\\ &\quad+\int_{0}^{t} e^{\left( (1+\beta_{i})k_{j1}-k_{i1}\right) W(r)+((1+\beta_{i})k_{j2}-k_{i2})B^{H}(r)+\gamma_i(t-r)} S_{t-r}(v_{j}(r,\cdot))^{1+\beta_{i}}(x) dr. \nonumber
   \end{align*}
	for $i=1,2, \ j\in \left\{ 1,2 \right\} / \left\{i\right\}.$
   \begin{theorem}\label{thm5.1}
  	Assume that $\beta_{1}\geq \beta_{2}>0,\ 	\gamma_{i}=\lambda+\frac{k_{i1}^{2}}{2},\ i=1,2$ and let $$f_{1}=C_{1}\psi \mbox{ and } f_{2}=C_{2}\psi,$$ for some positive constant $C_{1}$ and $C_{2}$ with $C_{1} \leq C_{2}.$ Then $\tau_{\ast \ast} \leq \tau,$ where $\tau_{\ast \ast}$ be given by 
  	\begin{align}
  	\tau_{**}= \inf \Bigg\{ t\geq 0 : \int^{t}_{0}  \exp\{((1+\beta_{1})k_{21}-k_{11})W(r)+&((1+\beta_{1})k_{22}-k_{12})B^{H}(r)\} dr\nonumber\\ 
  	&\geq \frac{1}{\beta_{1}C^{\beta_{1}}_{1}\left\|\psi \right\|_{\infty}^{\beta_{1}}}\nonumber\\ \ \mbox{or} \
  	\int^{t}_{0}  \exp\{((1+\beta_{2})k_{11}-k_{21})W(r)+&((1+\beta_{2})k_{12}-k_{22})B^{H}(r)\} dr\nonumber\\ 
  	&\geq  \frac{1}{\beta_{2}C^{\beta_{2}}_{2}\left\|\psi \right\|_{\infty}^{\beta_{2}}}\nonumber
  	\Bigg\}. \nonumber
  	\end{align}
  	\begin{proof}
  		For $i=1,2$, we define
  		\begin{align*}
  		\mathcal{T}_{i}v(t,x) &= \exp\{{\lambda t}\} S_{t}f_{i}(x)\nonumber\\
  		&\quad+\int_{0}^{t} e^{((1+\beta_{i})k_{j1}-k_{i1})W(r)+((1+\beta_{i})k_{j2}-k_{i2})B^{H}(r) +\gamma_{i} (t-r)} (S_{t-r} v(r,x))^{1+\beta_{i}} dr,
  		\end{align*}
  		and
  		$$\mathscr{G}_{i}(t)= \left[ 1- \beta_{i} \int_{0}^{t} \exp\{{((1+\beta_{i})k_{j1}-k_{i1}) W(r)+((1+\beta_{i})k_{j2}-k_{i2}) B^{H}(r)}\} \left\|e^{\lambda r} S_{r}f_{i} \right\|^{\beta_{i}}_{\infty} dr \right]^{\frac{-1}{\beta_{i}}}.$$
  		Proceeding in the same way as in the proof of Theorem $\ref{t2}$, we get  
  		\begin{eqnarray}
  		v_{1}(t,x) = \mathcal{T}_{1}v_{2}(t,x), \quad v_{2}(t,x) = \mathcal{T}_{2}v_{1}(t,x) ,\nonumber
  		\end{eqnarray}
  		whenever $t\leq \tau_{**},$ and $x\in D$. Moreover 
  		\begin{align}
  		v_{i}(t,x) \leq& \frac{\textcolor{red}{\exp\{{\lambda t}\} S_{t}f_{i}(x)}}{\left[ 1-\beta_{i} \int_{0}^{t} \exp\{{((1+\beta_{i})k_{j1}-k_{i1}) W(r)+((1+\beta_{i})k_{j2}-k_{i2}) B^{H}(r)}\} \left\| e^{\lambda r}S_{r}f_{i} \right\|^{\beta_{i}}_{\infty} dr  \right]^{\frac{1}{\beta_{i}}}} \nonumber\\
  		=& \frac{\textcolor{red}{C_{i}\psi(x)}}{\left[ 1- \beta_{i} C^{\beta_{i}}_{i} \left\|\psi \right\|^{\beta_{i}}_{\infty} \displaystyle\int_{0}^{t} \exp\{((1+\beta_{i})k_{j1}-k_{i1}) W(r)+((1+\beta_{i})k_{j2}-k_{i2}) B^{H}(r)\}  dr \right]^{\frac{1}{\beta_{i}}}}, \nonumber
  		\end{align}
  		by the choice of $f_{1}$ and $f_{2}$.
  	\end{proof}
  \end{theorem}
  \begin{corollary}\label{cor5.1}
  	Let the random time $\tau'$ be defined by 
  	\begin{align}
  	\tau'=\inf \Bigg\{ t\geq 0 : \int_{0}^{t} &\max\Big\{\displaystyle \exp\{((1+\beta_{i})k_{21}-k_{11})W(r)+((1+\beta_{i})k_{11}-k_{21})B^{H}(r)\},\nonumber\\&\qquad \displaystyle \exp\{((1+\beta_{i})k_{22}-k_{12})W(r)+((1+\beta_{i})k_{12}-k_{22})B^{H}(r)\} \Big\} dr \nonumber\\
  	&\geq \min \Bigg\{ \frac{1}{\beta_{1}C^{\beta_{1}}_{1}\left\|\psi \right\|_{\infty}^{\beta_{1}}},\frac{1}{\beta_{2}C^{\beta_{2}}_{2}\left\|\psi \right\|_{\infty}^{\beta_{2}}} \Bigg\} \Bigg\}. \nonumber
  	\end{align}
  	Then $\tau'\leq \tau_{\ast \ast}$.
  \end{corollary}
  To estimate upper bounds for $\tau,$ when $\beta_{1}\geq\beta_{2}>0,$ we first notice that a sub-solution of equation (\ref{a5}) is given by   
  \begin{equation}\label{ab1} 
  \left\{
  \begin{aligned}   \frac{dh_{1}(t)}{dt}&=(-\lambda+\gamma_{1}-k^{2})h_{1}(t)\\&\quad+\exp\{{((1+\beta_{1})k_{21}-k_{11})W(r)+((1+\beta_{1})k_{22}-k_{12})B^{H}(r)}\} h_{2}^{1+\beta_{1}}(t),\\ 
  \frac{dh_{2}(t)}{dt}&=(-\lambda+\gamma_{2}-k^{2})h_{2}(t)\\&\quad+\exp\{{((1+\beta_{2})k_{11}-k_{21})W(r)+((1+\beta_{2})k_{12}-k_{22})B^{H}(r)}\} h_{1}^{1+\beta_{2}}(t),\\
  h_{i}(0)&=v_{i}(0,\psi), \ i=1,2.  
  \end{aligned}
  \right.
  \end{equation}
  Implementing the same idea as in the proof of Theorem \ref{thm4.1} to the system (\ref{ab1}) with $\gamma=\min\{\gamma_{1}, \gamma_{2}\},$ $k^{2}=\max\left\lbrace  \frac{k_{11}^{2}}{2}, \frac{k_{21}^{2}}{2} \right\rbrace$ and if $\beta_{1}=\beta_{2}=\beta \ \mbox{(say)} >0$, we deduce that 
  \begin{align}
  \frac{dE(t)}{dt}&\geq (-\lambda+\gamma-k^{2})E(t)+ \mbox{min} \Big\{  \exp\{{((1+\beta)k_{21}-k_{11}) W(t)+((1+\beta)k_{22}-k_{12}) B^{H}(t)}\},\nonumber\\  &\qquad\qquad \exp\{{((1+\beta)k_{11}-k_{21}) W(t)+((1+\beta)k_{12}-k_{22}) B^{H}(t)}\} \Big\} 2^{-\beta} E^{1+\beta}(t). \nonumber
  \end{align} 
  If  $\beta_{1}>\beta_{2}>0,$ we obtain
  \begin{align}
  \frac{dE(t)}{E^{1+\beta_{2}}(t)}&\geq \frac{(-\lambda+\gamma-k^{2})}{E^{\beta_{2}}(t)}dt+ \mbox{min} \Big\{  \exp\{{((1+\beta_{1})k_{21}-k_{11}) W(t)+((1+\beta_{1})k_{22}-k_{12}) B^{H}(t)}\},\nonumber\\  &\qquad \exp\{{((1+\beta_{2})k_{11}-k_{21}) W(t)+((1+\beta_{2})k_{12}-k_{22}) B^{H}(t)}\} \Big\} \nonumber\\ &\qquad \times \left[\frac{\epsilon_{0}}{2^{1+\beta_{2}}}-\frac{D_{1}\epsilon_{0}^{\frac{1+\beta_{1}}{\beta_{1}-\beta_{2}}}}{E^{1+\beta_{2}}(0)} \right]dt,\nonumber
  \end{align} 
  which also follows from the proof of Theorem \ref{thm4.1}. In this manner we obtain an upper bound for the explosion time to the solution of the system \eqref{b1}-\eqref{b2} as follows:
  \begin{theorem}\label{thm5.2}
  	Suppose that $\beta_{1}\geq\beta_{2}>0$ and for any initial values $f=(f_{1}, f_{2})^{\top}$, we have the following results:
  	\item [1.] If $\beta_{1}=\beta_{2}=\beta \ \mbox{(say)}$ then $\tau \leq \tau^{\ast \ast}_{1}$, where 
  	\begin{align}
  	\tau^{\ast \ast}_{1}&= \inf \Bigg\{ t\geq 0 : \int_{0}^{t} \min \{e^{((1+\beta)k_{21}-k_{11}) B^{H}_{1}(s)+((1+\beta)k_{11}-k_{21}) B^{H}_{2}(s)+\beta(-\lambda+\gamma-k^{2})s},\nonumber\\  &\quad e^{((1+\beta)k_{22}-k_{12}) B^{H}_{1}(s)+((1+\beta)k_{12}-k_{22}) B^{H}_{2}(s)+\beta(-\lambda+\gamma-k^{2})s}\} ds\geq 2^{\beta} {\beta}^{-1}E^{-\beta}(0) \Bigg\}. \nonumber 
  	\end{align}  
  	\item [2.] 	If $\beta_{1}>\beta_{2}>0$ and let $D_{1}=\left( \frac{\beta_{1}-\beta_{2}}{1+\beta_{1}} \right)\left(\frac{1+\beta_{1}}{1+\beta_{2}}\right)^{\frac{1+\beta_{2}}{\beta_{1}-\beta_{2}}},$
  	\begin{align*}
  	\epsilon_{0}&\leq \min \Bigg\{ 1, \left( h_{2}(0)/D_{1}^{1/1+\beta_{2}} \right)^{\beta_{1}-\beta_{2}} \Bigg\}.
  	\end{align*} 
  	Assume that 
  	\begin{align*}
  	2^{-(1+\beta_{2})}\epsilon_{0}E^{1+\beta_{2}}(0)\geq \epsilon_{0}^{\frac{1+\beta_{1}}{\beta_{1}-\beta_{2}}}D_{1}.
  	\end{align*} 
  	Then $\tau \leq \tau^{\ast \ast}_{2},$ where 
  	\begin{align}
  	\tau^{\ast \ast}_{2} = \inf \Bigg\{ t\geq 0 : &\int_{0}^{t} \min \{e^{((1+\beta_{1})k_{21}-k_{11}) B^{H}_{1}(s)+((1+\beta_{1})k_{22}-k_{12}) B^{H}_{2}(s)+\beta_{2}(-\lambda+\gamma-k^{2})s},\nonumber\\  &\quad e^{((1+\beta_{2})k_{11}-k_{21}) B^{H}_{1}(s)+((1+\beta_{2})k_{12}-k_{22}) B^{H}_{2}(s)+\beta_{2}(-\lambda+\gamma-k^{2})s} \} ds \nonumber\\&\qquad\left.\geq {\left[\beta_{2} E^{\beta_{2}}(0)\left(\frac{\epsilon_0}{2^{1+\beta_{2}}}-\frac{ \epsilon_{0}^{\frac{1+\beta_{1}}{\beta_{1}-\beta_{2}}}D_{1}}{E^{1+\beta_{2}}(0)}\right)\right]^{-1}} \right\}.  \nonumber
  	\end{align} 
  	where  $\gamma=\min\{\gamma_{1}, \gamma_{2}\},$ $k^{2}=\max\left\lbrace  \frac{k_{11}^{2}}{2}, \frac{k_{21}^{2}}{2} \right\rbrace$ and $E(0)=\int_{D}[f_{1}(x)+f_{2}(x)]\psi(x)dx.$
  \end{theorem}

  \begin{theorem}\label{cor1}
  If $\beta_{1}\geq \beta_{2}>0,$ $\gamma_{i}=\lambda+\frac{k_{i1}^{2}}{2},\ i=1,2,$ and the initial values of the form $$f_{1}=C_{1}\psi \mbox{ and } f_{2}=C_{2}\psi,$$ for some positive constant $C_{1}$ and $C_{2}$ with $C_{1} \leq C_{2}$ such that 
  	\begin{align} {\label{cor3.1}}  
  	\beta_{i} \int_{0}^{\infty} \exp\{{((1+\beta_{i})k_{j1}-k_{i1})W(r)+((1+\beta_{i})k_{j2}-k_{i2})B^{H}(r)}\} \left\| \exp\{{\lambda r}\} S_{r}f_{i} \right\|_{\infty}^{\beta_{i}} dr <1, 
  	\end{align}
  	for $i=1,2, \ j\in \left\{ 1,2 \right\} / \left\{i\right\}.$ Then the weak solution $v=(v_{1},v_{2})^{\top}$ of the  system \eqref{s1} is global for all $(t,x) \in [0,\infty) \times D.$ Moreover, for all $(t,x) \in [0,\infty) \times D$ 
  	\begin{align*}
  	&0\leq v_{i}(t,x)\leq \\ &\displaystyle\frac{\exp\{{\lambda t}\} S_{t}f_{i}(x)}{\left[ 1-\beta_{i} \int_{0}^{t} \exp\{{((1+\beta_{i})k_{j1}-k_{i1})W(r)+((1+\beta_{i})k_{j2}-k_{i2})B^{H}(r)}\}  \left\| \exp\{{\lambda r}\} S_{r}f_{i} \right\|_{\infty}^{\beta_{i}} dr  \right]^{\frac{1}{\beta_{i}}}},
  	\end{align*}
  		for $i=1,2, \ j\in \left\{ 1,2 \right\} / \left\{i\right\}.$
  \end{theorem}
  \begin{proof}
  	Proof of this immediately follows from Theorem $\ref{thm5.1}.$  	
   \end{proof}	
	
   The following thoerem gives sharp bounds for $\left\lbrace p_{t}(x,y), \ t > 0 \right\rbrace$ (see \cite[Theorem 1.1]{wang1992}).
   \begin{theorem}\label{sha1}
	(\cite{wang1992}). Let $\psi$ be the first (normalized) Dirichlet eigenfunction of \eqref{a3}  on a connected bounded $C^{1, \alpha}$-domain $D \subset \mathbb{R}^{d}$, where $\alpha>0$ and $d \geq 1$, and let $\{p_{t}(x,y), \ t > 0 \}$ be the corresponding Dirichlet heat kernel. Then, there exists a constant $c > 0$ such that 
	\begin{align}\label{cd1}
	\max \left\lbrace 1,\frac{1}{c}t^{-\frac{(d+2)}{2}} \right\rbrace \leq e^{\lambda t}\sup_{x,y \in D} \frac{p_{t}(x,y)}{\psi(x) \psi(y)} \leq 1+c(1 \wedge t)^{-\frac{(d+2)}{2}}e^{-(\widehat{\lambda}-\lambda)t}, \ t>0,   
	\end{align}	
	where $\lambda,\widehat{\lambda}$ are the first two Dirichlet eigenvalues of  the Laplacian with $\widehat{\lambda}>\lambda$. This estimate is sharp for both short and long times.
    \end{theorem}
    The following result establishes the existence of a global weak solution $v=(v_{1},v_{2})^{\top}$ of the system \eqref{s1}   by using the sharp bounds which are given in Theorem \ref{sha1}.

   \begin{theorem}\label{sh1}
   	Let $\beta_{1}\geq \beta_{2}>0,$ $\gamma_{i}=\lambda+\frac{k_{i1}^{2}}{2},\ i=1,2,$ $D$ be a connected and bounded $C^{1, \alpha}$-domain in $\mathbb{R}^{d}$, where $\alpha>0$ and $d \geq 1$ and $\psi>0$ be the first eigenfunction on $D$. If the initial values are of the form $f_{i}(x) = C_{i} \psi(x), i=1,2$ for some positive constants $C_{1}$ and $C_{2}$ with $C_{1} \leq C_{2}$ satisfying
   	\begin{align} 
   	\beta_{i} \left( C_{2}(1+c)\|\psi\|_\infty^{2}\right)^{\beta_{i}} \int_{0}^{\infty} \exp\{{((1+\beta_{i})k_{j1}-k_{i1})W(r)+((1+\beta_{i})k_{j2}-k_{i2})B^{H}(r)}\} dr <1,\nonumber
   	\end{align}  
   	for $i=1,2, \ j\in \left\{ 1,2 \right\} / \left\{i\right\}.$ Then the solution $v=(v_1,v_2)^{\top}$ of the system \eqref{s1} exists globally. 
   \end{theorem}
   \begin{proof}
   	First, we let us verify the condition \eqref{cor3.1} of Theorem \ref{cor1}. For any $f_{i}(x) =C_{i} \psi(x), \ i=1,2$ and any $t>0,$ we have
   	\begin{align}
   		S_{t}f_{i}(x)&=\int_{D} p_{t}(x,y)f_{i}(y)dy =\int_{D} e^{\lambda t}\frac{p_{t}(x,y)}{\psi(x) \psi(y)}e^{-\lambda t}\psi(x) \psi(y) f_{i}(y)dy \nonumber\\	&\leq \left\| f_{i} \right\|_{\infty} \|\psi\|_\infty \int_{D} e^{\lambda t}\frac{p_{t}(x,y)}{\psi(x) \psi(y)}e^{-\lambda_{1}t} \psi(y) dy. \nonumber
   	\end{align} 
   	By using \eqref{cd1}, we have	
   	\begin{align} 
   		S_{t}f_{i}(x)& \leq \left\| f_{i} \right\|_{\infty} \|\psi\|_\infty \int_{D} \left( 1+c(1 \wedge t)^{-\frac{(d+2)}{2}}e^{-(\widehat{\lambda} -\lambda)t} \right) e^{-\lambda t} \psi(y) dy \nonumber\\
   		&\leq \left\| f_{i} \right\|_{\infty} \|\psi\|_\infty\left(e^{-\lambda t}+ce^{-\widehat{\lambda} t} \right)  \int_{D} \psi(y) dy \nonumber\\
   		&\leq \left\| f_{i}\right\|_{\infty} \|\psi\|_\infty\left(1+c \right)e^{-\lambda t}  \int_{D} \psi(y) dy. \nonumber
   	\end{align} 
   	Therefore,
   	\begin{align*}
   		\left\| e^{\lambda t}S_{t}f_{i}\right\|_{\infty} \leq C_{i}(1+c)\|\psi\|_\infty^{2}  \int_{D} \psi(y) dy\leq C_{2}(1+c)\|\psi\|_\infty^{2}, \ \ \mbox{for $i=1,2.$} 
   	\end{align*} 
   	The conditions given in \eqref{cor3.1} of Theorem \ref{cor1} are satisfied, provided
   	\begin{align*}
	\beta_{i} \left( C_{2}(1+c)\|\psi\|_\infty^{2}\right)^{\beta_{i}} \int_{0}^{\infty} \exp\{{((1+\beta_{i})k_{j1}-k_{i1})W(r)+((1+\beta_{i})k_{j2}-k_{i2})B^{H}(r)}\} dr <1, 
   	\end{align*} 
   	for $i=1,2, \ j\in \left\{ 1,2 \right\} / \left\{i\right\}$ and for  constants $C_1$ and $C_2$ sufficiently small. Rest of the proof follows from Theorem $\ref{cor1}$.
   \end{proof}

    \section{Probability of finite-time blow-up}\label{s2}
	In this section, we estimate the blow-up probability of solution $u=(u_1,u_2)^{\top}$ to the system \eqref{b1}-\eqref{b2}. We estimate a tail probability for the exponential function of Brownian motion and fBm when blow-up occurs before a fixed time $T > 0$. The method adopted here to estimate a tail probability for the exponential function of mixed Brownian motion and fBm are borrowed from \cite{doz2023} and \cite{dung}, respectively. 
	
	Let $B^{H}$	be a one-dimensional fBm defined on a complete probability space $\left( \Omega, \mathcal{F}, (\mathcal{F}_{t})_{t \geq 0} , \mathbb{P} \right).$ 
	For $h\in \mathbb{H}:=\mathrm{L}^{2}([0,T];\mathbb{R}),$ the Wiener integral $W(h)$ is denoted by 
	$$ W(h)=\int_{0}^{T} h(t) dW(t).$$ 
	Let $\mathcal{S}$ denote a dense subset of $\mathrm{L}^{2}(\Omega, \mathcal{F},\mathbb{P})$ consisting of smooth random variables of the form 
	\begin{align}\label{p1}
	F=f(W(h_{1}),W(h_{2}),\ldots, W(h_{n})), \ \mbox{for each} \ n\in \mathbb{N},
	\end{align}
	where $f \in C_{0}^{\infty}(\mathbb R^d)$ and for $i\in \{1,2,\ldots,n\}$, $h_{i} \in \mathbb{H}.$ If $F$ is of the form \textcolor{red}{given} in \eqref{p1},  we define its Malliavin derivative as the process 
	\begin{align} 
	DF:=(D_{t}F)_{t \in [0,T]}, \nonumber 
	\end{align} 
	given by 
	\begin{align*}
	D_{t}F = \sum_{k=1}^{n} \frac{\partial f}{\partial x_{k}} \left( W(h_{1}),W(h_{2}),\ldots,W(h_{n}) \right) h_{k}(t). 
	\end{align*}
	For any $1\leq p< \infty,$ we shall denote by $\mathbb{D}^{1,p}$, the closure of $\mathcal{S}$ with respect to the norm 
	\begin{align*}
	\|F\|_{1,p}^{p} := \mathbb{E}\left[| F |^{p}\right]+ \mathbb{E}\left[ \|\mathrm{D}F \|_{\mathbb{H}}^{p}\right], 
	\end{align*}   
	where $ \| \mathrm{D}F \|_{\mathbb{H}}^{p} = \int_{0}^{T} |\mathrm{D}_{t}F|^{p}dt $. A random variable $F$ is said to be \emph{Malliavin differentiable} if it belongs to $\mathbb{D}^{1,2}.$ 
	
	An fBm of Hurst parameter $H \in (0,1)$ is a centered Gaussian process $\{B^{H}(t)\}_{t \geq 0}$ with the covariance function (see \cite[Definition of 5.1, p.273]{nualart}) 
	\begin{align*}
	R_{H}(t,s):=\mathbb{E}\left[ B^{H}(t)B^{H}(s)\right] =\frac{1}{2} \left( s^{2H}+t^{2H}-|t-s|^{2H} \right),
	\end{align*}
so that $\mathbb{E}\left[|B^H(t)|^2\right]=t^{2H}$.	It is known that $B^{H}(t)$  admits the so called Volterra representation (for more details, see \cite{nualart}),
	\begin{align}
	B^{H}(t):=\int_{0}^{t} K^{H}(t,s) dW(s), \nonumber 
	\end{align}  
	where $\{W(t)\}_{t\geq 0},$ is a standard Brownian motion, the Volterra kernel $K_{H}(t,s)$ is defined by 
	\begin{align}
	K_{H}(t,s) = C_{H} \left[ \frac{t^{H-\frac{1}{2}}}{s^{H-\frac{1}{2}}}(t-s)^{H-\frac{1}{2}} -\left(H-\frac{1}{2}\right)\int_{s}^{t} \frac{u^{H-\frac{3}{2}}}{s^{H-\frac{1}{2}}}(u-s)^{H-\frac{1}{2}} du  \right], \ \ s \leq t, \nonumber 
	\end{align}
	where $C_{H}$ is a constant depending only on $H.$	By It\^o isometry, we also have 
	\begin{align*}
		\mathbb{E}\left[|B^H(t)|^2\right]=\int_0^t(K^H(t,s))^2ds.
	\end{align*}
Notice that $W$ and $B^{H}$ are dependent in this case. 
	
	The following theorem provides the tail probability of the exponential function which will be useful for further discussions.
	\begin{theorem}(\cite[Theorem 3.1]{dung2})\label{thme6.1} Let $\{X_{t}\}_{t \in [0,T]}$ be a continuous stochastic process in $\mathbb{D}^{1,2}.$ Assume that one of the following two conditions holds:
	\begin{itemize}
	\item[(i)] $\displaystyle\sup_{s \in [0,T]} \int_{0}^{T} |D_{r}X_{s}|^{2}dr \leq M^{2},\ \mathbb{P}\text{-a.s.},$ \nonumber\\
	\item[(ii)] $\displaystyle\int_{0}^{T} \mathbb{E} \left[ \displaystyle\sup_{s \in [0,T]} |D_{r}X_{s}|^{2}| \mathcal{F}_{r} \right] dr \leq M^{2},$ $\mathbb{P}$\text{-a.s.} 		
	\end{itemize}
	Then the tail probability of the exponential functional satisfies
	\begin{align}
    \mathbb{P}\left( \int_{0}^{T} e^{X_{s}} ds \geq x\right) \leq 2 e^{-\frac{(\ln x- \ln \mu)^{2}}{2 M^{2}}},\ x>\mu, 
	\end{align}
	where $\mu = \int_{0}^{T} \mathbb{E}[e^{X_{s}}]ds.$
	\end{theorem}
	
	The following theorem is useful to obtain an upper bound for the probability of finite-time blow-up	solution $u=(u_{1},u_{2})^{\top}$ and blow-up before a given fixed time $T>0$ of the system \eqref{b1}-\eqref{b2} with $\frac{1}{2}<H<1$.
	\begin{theorem}\label{thm6}
	Let $\beta_{1}\geq \beta_{2}>0$. For any non-negative initial data $f=(f_{1}, f_{2})^{\top}$, we have the following results:
	\begin{itemize}
	\item[1.] If $\beta_{1}=\beta_{2}=\beta$ (say) and let $\mu_{0}(T)=\displaystyle\int_{0}^{T}\exp\{-as\}\mathbb{E}\left( \exp\{\rho_{1}W(s)+\rho_{2}B^{H}(s)\}\right)  ds$, then for any $T>0$ such that $2^{\beta} {\beta}^{-1}E^{-\beta}(0)>\mu_{0}(T),$ 
	\begin{align}
	\mathbb{P}\left(\tau_{1}^{\ast} \leq T \right) \leq   2 \exp\left\lbrace -\frac{(\ln (2^{\beta} {\beta}^{-1}E^{-\beta}(0))- \ln(\mu_{0}(T)))^{2}}{2(2\rho_{1}^{2}T +2\rho_{2}^{2}T^{2H})^{2}}\right\rbrace, \nonumber 
	\end{align}
	where $a=\beta(\lambda-\gamma+k^{2}).$
	\item[2.] If $\beta_{1}>\beta_{2}$ and let $\mu_{1}(T)=\displaystyle\int_{0}^{T}\exp\{-a_{1}s\}\mathbb{E}\left( \exp\{\rho_{1}W(s)+\rho_{2}B^{H}(s)\}\right)  ds$, then for any $T>0$ such that $N>\mu_{1}(T),$
	\begin{align}
	\mathbb{P}\left(\tau_{2}^{\ast} \leq T \right) \leq 2 \exp\left\lbrace -\frac{(\ln \left({N} \right) - \ln(\mu_{1}(T)))^{2}}{2(2\rho_{1}^{2}T +2\rho_{2}^{2}T^{2H})^{2}}\right\rbrace, \nonumber 
	\end{align}
	\end{itemize}
	where $k^{2}=\max\left\lbrace  \frac{k_{11}^{2}}{2}, \frac{k_{21}^{2}}{2} \right\rbrace$,  $a_{1}=\beta_{2}(\lambda-\gamma+k^{2}),$  $N=\left[\beta_{2} E^{\beta_{2}}(0)\left(\frac{\epsilon_0}{2^{1+\beta_{2}}}-\frac{ \epsilon_{0}^{\frac{1+\beta_{1}}{\beta_{1}-\beta_{2}}}D_{1}}{E^{1+\beta_{2}}(0)}\right)\right]^{-1},$ $E(0)=\displaystyle\int_{D}[f_{1}(x)+f_{2}(x)]\psi(x)dx$, here $\epsilon_{0}, D_{1}$ and $\gamma$ are given in Theorem \ref{thm4.1}, and $\rho_{1}$, $\rho_{2}$ are defined in \eqref{a4}.
	\end{theorem}
	\begin{proof}
	For each $t \geq 0$, let 
	$$Z_{t}=-at+\rho_{1}W(t)+\rho_{2}B^{H}(t).$$
	The stochastic process $\{Z_{t}\}_{t \geq 0}$ satisfies the condition (i) of the Theorem \ref{thm6}. Since
	\begin{align}
	D_{r}Z_{t} &= \rho_{1}+\rho_{2}D_{r}\left(\int_{0}^{t} K^{H}(t,s)dW(s) \right)= \rho_{1}+\rho_{2}  K^{H}(t,r),\ \mbox{for}\ r \leq t,  \nonumber 
	\end{align}
	and $D_{r}Z_{t} =0,$ for $r>t.$ Further, we have
    \begin{align}
    \int_{0}^{T} |D_{r}Z_{t}|^{2}dr= \int_{0}^{t} \left[  \rho_{1}+\rho_{2}K^{H}(t, r)\right]^{2} dr &\leq    2\rho^{2}_{1}t+2\rho_{2}^{2}\int_{0}^{t}   (K^{H}(t,r))^{2}dr \nonumber\\
    &= 2\rho^{2}_{1}t+2\rho_{2}^{2}\mathbb{E}|B^{H}(t)|^{2}= 2\rho^{2}_{1}t+2\rho_{2}^{2}t^{2H}. \nonumber 
    \end{align}
    Therefore, we obtain
    \begin{align}\label{g1}
    \sup_{t \in [0,T]}\int_{0}^{T} |D_{r}X_{t}|^{2} dr \leq 2\rho_{1}^{2}T +2\rho_{2}^{2}T^{2H}.
    \end{align}
    If $\beta_{1}=\beta_{2}=\beta,$ then by using Theorem \ref{thm4.1} of case 1, we obtain that  $\tau \leq \tau^{\ast}_{1},$ where $\tau_{1}^{\ast}$ is given in \eqref{ST1} and using Theorem \ref{thme6.1}, we have
    \begin{align}
    \mathbb{P}\left( \tau^{\ast}_{1} \leq T\right)&= \mathbb{P}\left( \int_{0}^{T} \exp\{Z_{s}\} ds\geq 2^{\beta} {\beta}^{-1}E^{-\beta}(0) \right) \nonumber\\
    &\leq  2 \exp\left\lbrace -\frac{\left( \ln (2^{\beta} {\beta}^{-1}E^{-\beta}(0))- \ln(\mu_{0}(T))\right)^{2}}{2(\rho_{1}^{2}T +2\rho_{2}^{2}T^{2H})^{2}}\right\rbrace, \nonumber 
    \end{align}
    where $\mu_{0}(T)= \int_{0}^{T} \mathbb{E}\left(\exp\{Z_{s}\} \right) ds.$
    \vspace{.2 in} 
    \noindent \\ 
    If $\beta_{1}>\beta_{2}$, then by using Theorem \ref{thm4.1} of case 2, we obtain that $\tau \leq \tau^{\ast}_{2},$ where $\tau_{2}^{\ast}$ is given in \eqref{ST2}.  For each $t \geq 0$, let
    $$Z^{1}_{t}=-a_{1}t+\rho_{1}W(t)+\rho_{2}B^{H}(t).$$
    For $r\leq t,$ by using the same procedure as in  case 1, we have
    \begin{align}
    \mathbb{P}\left( \tau^{\ast}_{1} \leq T\right)&= \mathbb{P}\left( \int_{0}^{T} \exp\{Z^{1}_{s}\} ds\geq {\left[\beta_{2} E^{\beta_{2}}(0)\left(\frac{\epsilon_0}{2^{1+\beta_{2}}}-\frac{ \epsilon_{0}^{\frac{1+\beta_{1}}{\beta_{1}-\beta_{2}}}D_{1}}{E^{1+\beta_{2}}(0)}\right)\right]^{-1}} \right) \nonumber\\
    &\leq  2 \exp\left\lbrace -\frac{\left( \ln \left({\left[\beta_{2} E^{\beta_{2}}(0)\left(\frac{\epsilon_0}{2^{1+\beta_{2}}}-\frac{ \epsilon_{0}^{\frac{1+\beta_{1}}{\beta_{1}-\beta_{2}}}D_{1}}{E^{1+\beta_{2}}(0)}\right)\right]^{-1}} \right) - \ln(\mu_{1}(T))\right)^{2}}{2(\rho_{1}^{2}T +2\rho_{2}^{2}T^{2H})^{2}}\right\rbrace, \nonumber 
    \end{align}
    where $\mu_{1}(T)=\int_{0}^{T} \mathbb{E}(\exp\{Z^{1}_{t}\})dt,$ which completes the proof. 
	\end{proof}
	The following theorem provides  upper bounds for the tail probability of $\tau_{1}^{\ast}$ and $\tau_{2}^{\ast}$ with the general dependent structure of the Brownian motion and fBm.
	\begin{theorem}\label{thm6t}
	Let $\beta_{1}\geq \beta_{2}>0$, for any non-negative initial data $f=(f_{1}, f_{2})^{\top}$, we have the following results:
	\begin{itemize}
	\item[1.] Suppose $W$ and $B^{H}$ are dependent, say $B^{H}(t)=\int_{0}^{t} K^{H}(t,s)dW(s)$, where $W(\cdot)$ is a Brownian motion defined in the same probability space.\\
	(a) If $\beta_{1}=\beta_{2}=\beta$ with $\rho_{1}^{2}>\beta(\lambda-\gamma+k^{2})$, then
	\begin{align}
	\mathbb{P}(\tau^{\ast}_{1}\leq T) &\leq 2^{-(\beta+1)} {\beta}E^{\beta}(0) \Bigg[ \frac{\exp\{\left( \rho^{2}_{1}+\beta(-\lambda+\gamma-k^{2})\right) T\}-1}{\left( \rho^{2}_{1}+\beta(-\lambda+\gamma-k^{2})\right) } \nonumber\\
	&\qquad\qquad+  \int_{0}^{T} \exp\{\beta(-\lambda+\gamma-k^{2})s+2 \rho_{2}^{2}s^{2H} \}ds\Bigg]. \nonumber 
	\end{align}
	(b) If $\beta_{1}>\beta_{2}$ with $\rho_{1}^{2}>\beta_{2}(\lambda-\gamma+k^{2})$, then 
	\begin{align}
	\mathbb{P}(\tau^{\ast}_{2}\leq T) &\leq  \left[\beta_{2} E^{\beta_{2}}(0)\left(\frac{\epsilon_0}{2^{1+\beta_{2}}}-\frac{ \epsilon_{0}^{\frac{1+\beta_{1}}{\beta_{1}-\beta_{2}}}D_{1}}{E^{1+\beta_{2}}(0)}\right)\right]\Bigg[ \frac{\exp\{\left( \rho^{2}_{1}+\beta_{2}(-\lambda+\gamma-k^{2})\right) T\}-1}{\left( \rho^{2}_{1}+\beta_{2}(-\lambda+\gamma-k^{2})\right) } \nonumber\\
	&\qquad\qquad+  \int_{0}^{T} \exp\{\beta_{2}(-\lambda+\gamma-k^{2})s+2 \rho_{2}^{2}s^{2H} \}ds\Bigg].\nonumber  
	\end{align} 
	\item[2.] If $W$ and $B^{H}$ are independent, then \\
	(a)  for $\beta_{1}=\beta_{2}=\beta,$ we have
	\begin{align}
	\mathbb{P}(\tau^{\ast}_{1}\leq T)\leq  2^{-\beta} {\beta}E^{\beta}(0) \int_{0}^{T} \exp \left\lbrace \left( \frac{\rho^{2}_{1}}{2} +\beta(-\lambda+\gamma-k^{2})\right) s+\frac{\rho_{2}^{2}}{2} s^{2H}\right\rbrace ds, \nonumber
	\end{align}	
	(b) for $\beta_{1}>\beta_{2},$ we have 
	\begin{align}
	\mathbb{P}(\tau^{\ast}_{2}\leq T)
	&\leq \left[\beta_{2} E^{\beta_{2}}(0)\left(\frac{\epsilon_0}{2^{1+\beta_{2}}}-\frac{ \epsilon_{0}^{\frac{1+\beta_{1}}{\beta_{1}-\beta_{2}}}D_{1}}{E^{1+\beta_{2}}(0)}\right)\right]\nonumber\\
	&\qquad\qquad \times \int_{0}^{T} \exp \left\lbrace \left( \frac{\rho^{2}_{1}}{2} +\beta_{2}(-\lambda+\gamma-k^{2})\right) s+\frac{\rho_{2}^{2}}{2}s^{2H} \right\rbrace  ds, \nonumber 
	\end{align}
 	\end{itemize}
 	where $k^{2}=\max\left\lbrace  \frac{k_{11}^{2}}{2}, \frac{k_{21}^{2}}{2} \right\rbrace$,  $E(0)=\displaystyle\int_{D}[f_{1}(x)+f_{2}(x)]\psi(x)dx$, here $\epsilon_{0}, D_{1}$ and $\gamma$ are given in Theorem \ref{thm4.1}, and $\rho_{1}$, $\rho_{2}$ are defined in \eqref{a4}.
	\end{theorem}
	\begin{proof}
	\textbf{Case 1:} Assume that $B^{H}(t)=\int_{0}^{t} K^{H}(t,s)dW(s)$, where $W(\cdot)$ is a Brownian motion defined in the same probability space, and adapted to the same filtration as the fractional Brownian motion $B^H(\cdot)$. If $\beta_{1}=\beta_{2}=\beta,$ then by using Theorem \ref{thm4.1} case 1, we obtain $\tau \leq \tau_{1}^{\ast},$ where $\tau_{1}^{\ast}$ is given in \eqref{ST1}, and by using H\"older's and Markov's  inequalities, we obtain
	\begin{align}\label{c2}
	\mathbb{P}(\tau^{\ast}_{1}\leq T)&= \mathbb{P}\left( \int_{0}^{T} \exp\{\rho_{1} W(s)+\rho_{2} B^{H}(s)+\beta(-\lambda+\gamma-k^{2})s\} ds\geq 2^{\beta} {\beta}^{-1}E^{-\beta}(0) \right) \nonumber\\
	& \leq \mathbb{P}\Bigg[  \int_{0}^{T} \left( \exp\{2\rho_{1} W(s)+\beta(-\lambda+\gamma-k^{2})s\} ds\right)^{\frac{1}{2}} \nonumber\\
	&\qquad \times \int_{0}^{T} \left( \exp\{2\rho_{2} B^{H}(s)+\beta(-\lambda+\gamma-k^{2})s\} ds\right)^{\frac{1}{2}}\geq 2^{\beta} {\beta}^{-1}E^{-\beta}(0) \Bigg] \nonumber\\
	& \leq \mathbb{P}\Bigg[  \int_{0}^{T}  \exp\{2\rho_{1} W(s)+\beta(-\lambda+\gamma-k^{2})s\} ds\geq 2^{\beta+1} {\beta}^{-1}E^{-\beta}(0)\Bigg] \nonumber\\
	&\qquad +\mathbb{P}\Bigg[ \int_{0}^{T}  \exp\{2\rho_{2} B^{H}(s)+\beta(-\lambda+\gamma-k^{2})s\} ds\geq 2^{\beta+1} {\beta}^{-1}E^{-\beta}(0) \Bigg] \nonumber\\
	&\leq \frac{\mathbb{E}\left[ \int_{0}^{T}  e^{2\rho_{1} W(s)+\beta(-\lambda+\gamma-k^{2})s} ds \right]+ \mathbb{E}\left[ \int_{0}^{T}  e^{2\rho_{2} B^{H}(s)+\beta(-\lambda+\gamma-k^{2})s} ds \right]}{2^{\beta+1} {\beta}^{-1}E^{-\beta}(0)} \nonumber\\
	&\leq \frac{ \int_{0}^{T}  \exp\{\rho_{1}^{2}s+\beta(-\lambda+\gamma-k^{2})s\} ds +  \int_{0}^{T} e^{\beta(-\lambda+\gamma-k^{2})s} \mathbb{E}\left[ \exp\{2\rho_{2} B^{H}(s)\}\right] ds }{2^{\beta+1} {\beta}^{-1}E^{-\beta}(0)}.	
	\end{align}
	Let $\bar{Z}(t)=\displaystyle\int_{0}^{t}f(s)dW(s).$ By applying Ito's formula to the process $\left\{ \exp\{\bar{Z}(t)\}\right\}_{t \geq 0}$, we have
	\begin{align}
	\exp\{\bar{Z}(t)\}=1+\int_{0}^{t} \exp\{\bar{Z}(s)\} d\bar{Z}(s)+\frac{1}{2}\int_{0}^{t} \exp\{\bar{Z}(s)\}f^{2}(s)ds. \nonumber
	\end{align}
	Taking expectation on both sides, we get
	\begin{align}
	\mathbb{E}(\exp\{\bar{Z}(t)\})=1+\frac{1}{2}\int_{0}^{t} \mathbb{E}\left( \exp\{\bar{Z}(s)\}\right)f^{2}(s)ds. \nonumber 
	\end{align} 
    Therefore by taking $\bar{Y}(t):=\mathbb{E}(\exp\{\bar{Z}(t)\})$, and variation of constants formula  yields
    \begin{align}\label{f2}
    \mathbb{E}\left(\exp\left\lbrace  \int_{0}^{t} f(s)dW(s) \right\rbrace  \right)=\exp \left\lbrace \frac{1}{2}\int_{0}^{t} f^{2}(s)ds\right\rbrace.  
    \end{align}    
	Using \eqref{f2}, it follows that 
	\begin{align}
	\mathbb{E}\left[e^{2\rho_{2} B^{H}(s)}  \right] = \mathbb{E}\left[ e^{2\rho_{2} \int_{0}^{s} K^{H}(s,r)dW(s)}\right] &= \exp\{2 \rho_{2}^{2}\int_{0}^{s} (K^{H}(s,r))^{2} ds\} \nonumber\\
	&= \exp\{2 \rho_{2}^{2}\mathbb{E}(|B^{H}_{s}|^{2})\}=\exp\{2 \rho_{2}^{2}s^{2H} \}.\nonumber 	
	\end{align}
	Therefore, from \eqref{c2}, we get
	\begin{align}
	\mathbb{P}(\tau^{\ast}_{1}\leq T) &\leq 2^{-(\beta+1)} {\beta}E^{\beta}(0) \Bigg[ \int_{0}^{T} \exp\{\left( \rho^{2}_{1}+\beta(-\lambda+\gamma-k^{2})\right) s\}ds \nonumber\\
	&\qquad\qquad+  \int_{0}^{T} \exp\{\beta(-\lambda+\gamma-k^{2})s+2 \rho_{2}^{2}s^{2H} \}ds\Bigg] \nonumber\\
	&=2^{-(\beta+1)} {\beta}E^{\beta}(0) \Bigg[ \frac{\exp\{\left( \rho^{2}_{1}+\beta(-\lambda+\gamma-k^{2})\right) T\}-1}{\left( \rho^{2}_{1}+\beta(-\lambda+\gamma-k^{2})\right) } \nonumber\\
	&\qquad\qquad+  \int_{0}^{T} \exp\{\beta(-\lambda+\gamma-k^{2})s+2 \rho_{2}^{2}s^{2H} \}ds\Bigg]. \nonumber
	\end{align}
	Similarly, if $\beta_{1}>\beta_{2},$ then by using Theorem \ref{thm4.1} case 2, we obtain $\tau \leq \tau_{2}^{\ast},$ where $\tau_{2}^{\ast}$ is given in \eqref{ST2}, and 
	\begin{align}
	\mathbb{P}(\tau^{\ast}_{2}\leq T) &\leq  \left[\beta_{2} E^{\beta_{2}}(0)\left(\frac{\epsilon_0}{2^{1+\beta_{2}}}-\frac{ \epsilon_{0}^{\frac{1+\beta_{1}}{\beta_{1}-\beta_{2}}}D_{1}}{E^{1+\beta_{2}}(0)}\right)\right]\Bigg[ \frac{\exp\{\left( \rho^{2}_{1}+\beta_{2}(-\lambda+\gamma-k^{2})\right) T\}-1}{\left( \rho^{2}_{1}+\beta_{2}(-\lambda+\gamma-k^{2})\right) } \nonumber\\
	&\qquad\qquad+  \int_{0}^{T} \exp\left\lbrace \beta_{2}(-\lambda+\gamma-k^{2})s+2 \rho_{2}^{2}s^{2H} \right\rbrace  ds\Bigg].\nonumber  
	\end{align} 
	\textbf{Case 2:} If $\beta_{1}=\beta_{2}=\beta$, $W$ and $B^{H}$ are independent, then by using Markov's inequality,  we obtain 
	\begin{align}
	\mathbb{P}(\tau^{\ast}_{1}\leq T)&= \mathbb{P}\left( \int_{0}^{T} \exp\{\rho_{1} W(s)+\rho_{2} B^{H}(s)+\beta(-\lambda+\gamma-k^{2})s\} ds\geq 2^{\beta} {\beta}^{-1}E^{-\beta}(0) \right) \nonumber\\
	&\leq 2^{-\beta} {\beta}E^{\beta}(0) \int_{0}^{T} \mathbb{E}\left(\exp\{\rho_{1} W(s)+\beta(-\lambda+\gamma-k^{2})s\}\right) \mathbb{E}\left(\exp\{\rho_{2} B^{H}(s)\} \right)  ds \nonumber\\
	&= 2^{-\beta} {\beta}E^{\beta}(0) \int_{0}^{T} \exp \left\lbrace \left( \frac{\rho^{2}_{1}}{2} +\beta(-\lambda+\gamma-k^{2})\right) s+ \frac{\rho_{2}^{2}}{2}s^{2H}\right\rbrace  ds. \nonumber 
	\end{align}
	Similarly, if $\beta_{1}>\beta_{2}$ and $W$ and $B^{H}$ are independent, then 
	\begin{align}
	\mathbb{P}(\tau^{\ast}_{2}\leq T)
	&\leq \left[\beta_{2} E^{\beta_{2}}(0)\left(\frac{\epsilon_0}{2^{1+\beta_{2}}}-\frac{ \epsilon_{0}^{\frac{1+\beta_{1}}{\beta_{1}-\beta_{2}}}D_{1}}{E^{1+\beta_{2}}(0)}\right)\right]\nonumber\\
	&\qquad\qquad \times \int_{0}^{T} \exp \left\lbrace \left( \frac{\rho^{2}_{1}}{2} +\beta_{2}(-\lambda+\gamma-k^{2})\right) s+ \frac{\rho_{2}^{2}}{2}s^{2H} \right\rbrace  ds, \nonumber 
	\end{align}
	which completes the proof.
	\end{proof}

	\subsection{A lower bound for the probability of finite-time blow-up}
	In this section, we estimate a lower bound for the probability of finite-time blow-up of weak solution $u=(u_1,u_2)^{\top}$ of the system \eqref{b1}-\eqref{b2}. 
	The following theorem is very useful for our further discussions. 
	\begin{assumption} \label{as2}
	The stochastic process $\left(X_{t}\right)_{t \geq 0}$ is $\mathcal{F}_t$-adapted and satisfies the following properties:   
	\end{assumption}
	\begin{itemize}	
	\item [1.] $\displaystyle\int_{0}^{\infty} \exp\{-as\}\mathbb{E}\left[ \exp\{\sigma_{1}X_{s}\}\right]  ds < +\infty.$
	\item [2.] For each $t \geq 0, \ X_{t} \in \mathbb{D}^{1,2}.$
	\item [3.] There exists a function $f : \mathbb{R}^{+} \rightarrow \mathbb{R}^{+}$ such that $\displaystyle\lim_{t \rightarrow \infty} f(t)= +\infty$ and for each $x>0,$
	\begin{align*}
	\sup\limits_{t\geq 0} \frac{\sup\limits_{s\in [0,t]} \int_{0}^{s} |D_{\theta_{1}}X_{s}|^{2} d\theta_{1}}{\left( \ln \left( x+1\right)+f(t) \right)^{2}} \leq M_{x}< +\infty, \ \ \mathbb{P}\text{-a.s.}
	\end{align*}
	\end{itemize}
	The following result is very useful to estimate the probability of blow-up of positive solution $u=(u_{1},u_{2})^{\top}$ of the system \eqref{b1}-\eqref{b2}.
	\begin{theorem}(\cite[Theorem 3.1]{dung})\label{thm6.4}
	Suppose that Assumption  \ref{as2} holds. Then, we have 
	\begin{align}
	\mathbb{P} \left[ \int_{0}^{\infty} \exp\{-as+\sigma_{1}X_{s} \} ds < x \right] \leq \exp \left\lbrace -\frac{(m_{x} -1)^{2}}{2 \sigma^{2}M_{x}}\right\rbrace,
	\end{align}
	where 
	\begin{align*}
	m_{x}=\mathbb{E}\left[\sup\limits_{t \geq 0} \frac{\ln \left(  \int_{0}^{t} \exp\{-as+\sigma_{1}X_{s}+\sigma_{2}Y_{s} \} ds+1 \right) +f(t)}{\ln \left( x+1\right)+f(t)}\right]\geq 1.
	\end{align*}
	\end{theorem}
	The following theorem provides a lower bound for the probability of blow-up solution $u=(u_1,u_2)^{\top}$ of the system \eqref{b1}-\eqref{b2}.
	\begin{theorem}\label{thm6.5}
	Let $\{W(t)\}_{t\geq 0}$ and $\{B^{H}(t)\}_{t\geq 0}$ be a standard Brownian motion and fBm with Hurst parameter $\frac{1}{2}<H<1,$  respectively. For any $\alpha > H,$ it holds that for any initial values $f=(f_{1}, f_{2})^{\top}$, we have the following results:
	\begin{itemize}
	\item [1.] If $\beta_{1}=\beta_{2}=\beta$ (say), then a lower bound for the probability of blow-up solution $u=(u_1,u_2)^{\top}$ of the system \eqref{b1}-\eqref{b2} is given by  
	\begin{align}\label{PB1}
	&\mathbb{P} \left(\tau< \infty \right)= 1- \mathbb{P}(\tau=\infty) \nonumber\\ &\geq  
	1-\exp \left\lbrace -\frac{\alpha^{2}(L_{1}(\alpha) -1)^{2}}{\rho_{1}^{2}(2\alpha-1)^{2-\frac{1}{\alpha}}\ln(U+1)^{\frac{1}{\alpha}-2}+2 \rho_{2}^{2}\alpha^{2}\ln(U+1)^{\frac{2H}{\alpha}-2}\left( \frac{\alpha-H}{\alpha}\right)^{2-\frac{2H}{\alpha}}}\right\rbrace, 
	\end{align}
	where $\tau \leq \tau_{1}^{\ast},$  $\tau_{1}^{\ast}$ is given in \eqref{ST1},
	\begin{align} \label{N1}
	L_{1}(\alpha):= \mathbb{E}\left[\sup\limits_{t \geq 0} \frac{\ln \left(  \int_{0}^{t} \exp\{-\beta(\lambda-\gamma+k^{2})s+\rho_{1}W(s)+\rho_{2}B^{H}(s) \} ds+1 \right) +t^{\alpha}}{\ln \left( 2^{\beta} {\beta}^{-1}E^{-\beta}(0)+1\right)+t^{\alpha}}\right], 
	\end{align}
	$a=\beta(\lambda-\gamma+k^{2}),$ and $U= 2^{\beta} {\beta}^{-1}E^{-\beta}(0).$
	\item [2.] If $\beta_{1}>\beta_{2}$, then a lower bound for the probability of blow-up solution $u=(u_1,u_2)^{\top}$ of the system \eqref{b1}-\eqref{b2} is given by
	\begin{align}\label{PB2}
	&\mathbb{P} \left(\tau< \infty \right)=1-\mathbb{P}(\tau = \infty)\geq  \nonumber\\
	& 1-\exp \left\lbrace -\frac{\alpha^{2}(L_{2}(\alpha) -1)^{2}}{\rho_{1}^{2}(2\alpha-1)^{2-\frac{1}{\alpha}}\ln(N+1)^{\frac{1}{\alpha}-2}+2 \rho_{2}^{2}\alpha^{2}\ln(N+1)^{\frac{2H}{\alpha}-2}\left( \frac{\alpha-H}{\alpha}\right)^{2-\frac{2H}{\alpha}}}\right\rbrace, 
	\end{align} 
	where $\tau \leq \tau_{2}^{\ast},$  $\tau_{2}^{\ast}$ is given in \eqref{ST2},
	\begin{align}\label{N2} 
	L_{2}(\alpha):= \mathbb{E}\left[\sup\limits_{t \geq 0} \frac{\ln \left(  \int_{0}^{t} \exp\{-\beta(\lambda-\gamma+k^{2})s+\rho_{1}W(s)+\rho_{2}B^{H}(s) \} ds+1 \right) +t^{\alpha}}{\ln \left( \left[\beta_{2} E^{\beta_{2}}(0)\left(\frac{\epsilon_0}{2^{1+\beta_{2}}}-\frac{ \epsilon_{0}^{\frac{1+\beta_{1}}{\beta_{1}-\beta_{2}}}D_{1}}{E^{1+\beta_{2}}(0)}\right)\right]^{-1}+1\right)+t^{\alpha}}\right], 
	\end{align}  
	 $k^{2}=\max\left\lbrace  \frac{k_{11}^{2}}{2}, \frac{k_{21}^{2}}{2} \right\rbrace$, $N=\left[\beta_{2} E^{\beta_{2}}(0)\left(\frac{\epsilon_0}{2^{1+\beta_{2}}}-\frac{ \epsilon_{0}^{\frac{1+\beta_{1}}{\beta_{1}-\beta_{2}}}D_{1}}{E^{1+\beta_{2}}(0)}\right)\right]^{-1}$ and $a_{1}=\beta_{2}(\lambda-\gamma+k^{2})$, $E(0)=\displaystyle\int_{D}[f_{1}(x)+f_{2}(x)]\psi(x)dx$, here $\epsilon_{0}, D_{1}$ and $\gamma$ are given in Theorem \ref{thm4.1}, and $\rho_{1}$, $\rho_{2}$ are defined in \eqref{a4}.
	\end{itemize}
	
	\end{theorem}
	\begin{proof}
	We have 
	\begin{align}
	\mathbb{P} \left( \int_{0}^{\infty} e^{-as +\rho_{1}W(s)+\rho_{2}B^{H}(s)}ds \right) \leq  \mathbb{P} \left( \int_{0}^{\infty} e^{-as-\frac{\rho_{1}^{2}}{2}s +\rho_{1}W(s)-\frac{\rho_{2}^{2}}{2}s^{2H}+\rho_{2}B^{H}(s)}ds \right). \nonumber 	\end{align}	
	We choose for each $t \geq 0$, and using \eqref{ST1}, we have the process representation
	\begin{align}
	\hat{Z}_{t}:&=  \rho_{1}X_{t}+\rho_{2}Y_{t}-at, \nonumber
	\end{align}
	where $X_{t}=-\frac{\rho_{1}}{2}t+W(t)$ and $Y_{t}=-\frac{\rho_{2}}{2}t^{2H}+B^{H}(t)$. The stochastic process $\{\hat{Z}_{t}\}_{t \geq 0},$ satisfies Assumption \ref{as2} of Theorem \ref{thm6.4}, since
	\begin{align}
	\int_{0}^{\infty}\mathbb{E} \left(  e^{-as-\frac{\rho_{1}^{2}}{2}s +\rho_{1}W(s)-\frac{\rho_{2}^{2}}{2}s^{2H}+\rho_{2}B^{H}(s)}\right) ds &=\int_{0}^{\infty} e^{-as-\frac{\rho_{1}^{2}}{2}s -\frac{\rho_{2}^{2}}{2}s^{2H}} \mathbb{E}\left( e^{\rho_{1}W(s)+\rho_{2}B^{H}(s)}\right)  ds\nonumber\\
	&=\int_{0}^{\infty} e^{-as} ds<\infty, \nonumber  
	\end{align}
	and for $r \leq t,$ we have
	\begin{align}
	D_{r}\hat{Z}_{t} = \rho_{1}+\rho_{2}D_{r}B^{H}(t) &= \rho_{1}+\rho_{2}D_{r}\left( \int_{0}^{t}K^{H}(t, s)dW(s)\right) = \rho_{1}+\rho_{2}K^{H}(t, r).  \nonumber
	\end{align}
	Further, we have
	\begin{align}
	\sup_{s \in [0,t]}\int_{0}^{s} |D_{r}Z_{s}|^{2}dr&= \sup_{s \in [0,t]}\int_{0}^{s} \left[  \rho_{1}+\rho_{2}K^{H}(s, r)\right]^{2} dr\nonumber\\ &\leq    2\rho^{2}_{1}t+2\rho_{2}^{2}\sup_{s \in [0,t]}\int_{0}^{s}   (K^{H}(s,r))^{2}dr \nonumber\\
	&= 2\rho^{2}_{1}t+2\rho_{2}^{2}\sup_{s \in [0,t]}\mathbb{E}|B^{H}(s)|^{2}= 2\rho^{2}_{1}t+2\rho_{2}^{2}t^{2H}. \nonumber 
	\end{align}
	From \eqref{g1}, we have
	\begin{align} 
	\sup_{s \in [0,t]}\int_{0}^{s} |D_{r}X_{s}|^{2} \leq 2\rho_{1}^{2}t +2\rho_{2}^{2}t^{2H}.\nonumber 
	\end{align}
	Hence, $\hat{Z}_{t}$ also satisfies the condition (iii) of Assumption \ref{as2} with $f(t)=t^{\alpha}, \alpha>H$ and $x>0,$
	\begin{align}
	M_{x}&:= \sup_{t \geq 0} \frac{2\rho_{1}^{2}t +2\rho_{2}^{2}t^{2H}}{(\ln(x+1)+t^{\alpha})^{2}}\nonumber\\
	&\leq \sup_{t \geq 0} \frac{2\rho_{1}^{2}t }{(\ln(x+1)+t^{\alpha})^{2}}+\sup_{t \geq 0} \frac{2\rho_{2}^{2}t^{2H}}{(\ln(x+1)+t^{\alpha})^{2}}.\nonumber	
	\end{align}
	Therefore, we get
	\begin{align}
	M_{x}\leq  \frac{\rho_{1}^{2}(2\alpha-1)^{2-\frac{1}{\alpha}}}{\alpha^{2}}\ln(x+1)^{\frac{1}{\alpha}-2}+2 \rho_{2}^{2}\ln(x+1)^{\frac{2H}{\alpha}-2}\left( \frac{\alpha-H}{\alpha}\right)^{2-\frac{2H}{\alpha}}. \nonumber  
	\end{align}
	Hence by Theorem \ref{thm6.4}, we obtain
	\begin{align}
	&\mathbb{P} \left[ \int_{0}^{\infty} \exp\{-\beta(\lambda-\gamma+k^{2})s+\rho_{1}W(s)+\rho_{2}B^{H}(s) \} ds < 2^{\beta} {\beta}^{-1}E^{-\beta}(0) \right] \nonumber\\
	&\qquad\leq \exp \left\lbrace \frac{-\alpha^{2} (L_{1}(\alpha) -1)^{2}}{\rho_{1}^{2}(2\alpha-1)^{2-\frac{1}{\alpha}}\ln(U+1)^{\frac{1}{\alpha}-2}+2 \alpha^{2} \rho_{2}^{2}\ln(U+1)^{\frac{2H}{\alpha}-2}\left( \frac{\alpha-H}{\alpha}\right)^{2-\frac{2H}{\alpha}}}\right\rbrace, \nonumber 
	\end{align}
	where $L_{1}(\alpha)$ is given in \eqref{N1}, and the desired bound is given in \eqref{PB1}.
	\vspace{0.2 in}
	\noindent\\
	\textbf{Case 2:} If $\beta_{1}>\beta_{2},$ we obtain $\tau \leq \tau_{2}^{\ast}$, where $\tau_{2}^{\ast}$ is given in \eqref{ST2}, by definition of $\tau_{2}^{\ast},$ we have
	\begin{align}
	\mathbb{P} \left( \int_{0}^{\infty} e^{-a_{1}s +\rho_{1}W(s)+\rho_{2}B^{H}(s)}ds \right) \leq  \mathbb{P} \left( \int_{0}^{\infty} e^{-a_{1}s-\frac{\rho_{1}^{2}}{2}s +\rho_{1}\int_{0}^{s}dW(r)-\frac{\rho_{2}^{2}}{2}s^{2H}+\rho_{2}\int_{0}^{s}dB^{H}(r)}ds \right). \nonumber 	
	\end{align}	
	For each $t \geq 0$, we define
	\begin{align}
	\hat{Z}^{1}_{t}:&=  \rho_{1}X_{t}+\rho_{2}Y_{t}-a_{1}t, \nonumber
	\end{align}
	where $X_{t}=-\frac{\rho_{1}}{2}t+W(t)$ and $Y_{t}=-\frac{\rho_{2}}{2}t^{2H}+B^{H}(t)$. By using the same procedure as in Case 1, the stochastic process $\{\hat{Z}^{1}_{t}\}_{t\geq0},$ satisfies Assumption \ref{as2} of Theorem \ref{thm6.4}, we have
	\begin{align}
	\mathbb{P} &\left[ \int_{0}^{\infty} e^{-\beta_{2}(\lambda-\gamma+k^{2})s+\rho_{1}W(s)+\rho_{2}B^{H}(s) } ds < N \right] \nonumber\\
	&\hspace{0.4 in} \leq \exp \left\lbrace -\frac{\alpha^{2}(L_{2}(\alpha) -1)^{2}}{\rho_{1}^{2}(2\alpha-1)^{2-\frac{1}{\alpha}}\ln(N+1)^{\frac{1}{\alpha}-2}+2\alpha^{2} \rho_{2}^{2}\ln(N+1)^{\frac{2H}{\alpha}-2}\left( \frac{\alpha-H}{\alpha}\right)^{2-\frac{2H}{\alpha}}}\right\rbrace, \nonumber 
	\end{align}
    where $N=\left[\beta_{2} E^{\beta_{2}}(0)\left(\frac{\epsilon_0}{2^{1+\beta_{2}}}-\frac{ \epsilon_{0}^{\frac{1+\beta_{1}}{\beta_{1}-\beta_{2}}}D_{1}}{E^{1+\beta_{2}}(0)}\right)\right]^{-1}$ and this completes the proof.
	\end{proof}
	
	\subsection{The case $H>\frac{3}{4}$ and independent $W$ and $B^{H}$} 
	In this subsection, we consider the case $\frac{3}{4}<H<1$ and  $W$ and $B^{H}$ are independent with $\rho_{1}=\rho_{2}=\rho$ (say), where $\rho_{1}$ and $\rho_{2}$ are defined in \eqref{a4}. Following the ideas from \cite[Subsection 4.2]{doz2023}, we obtain more explicit lower bound for the probability of finite-time blow-up of weak solution $u=(u_1,u_2)^{\top}$ of the system \eqref{b1}-\eqref{b2}.  Taking $M(t)=W(t)+B^{H}(t)$, by \cite[Theorem 1.7]{cher2001}, $M$ is equivalent to a Brownian motion $\widetilde{B}.$ Here the equivalence means equality of the laws of the processes on $(C[0,T],\mathcal{B}),$ the space of continuous functions defined on $[0,T]$ endowed with the $\sigma$- algebra generated by the cylinder sets. 
	
	We recall that $X(\alpha_{1},\delta_{1})$ is said to be \emph{a Gamma random variable} with parameters $\alpha_{1}+1>0,\ \delta_{1}>0$ if its density is given by (cf. \cite{li})
	\begin{equation} \label{abc2}
	\widetilde{f}(x) = \left\{
	\begin{aligned}
	&\frac{x^{\alpha_{1}}}{\delta_{1}^{\alpha_{1}+1}\Gamma (\alpha_{1}+1)} \exp\left\lbrace-\frac{x}{\delta_{1}} \right\rbrace , \ x\geq 0 , \\
	&\hspace{.5 in} 0, \hspace{1.05 in} \ x<0. 
	\end{aligned}
	\right.
	\end{equation}
	
	\begin{theorem} \label{thm6.6}
    For each positive initial values $f=(f_{1}, f_{2})^{\top}$ and positive real numbers $ \epsilon_{0}, D_{1}$ defined in Theorem \ref{thm4.1}, we have the following results:
    \begin{itemize}
    \item[1.] If $\beta_{1}=\beta_{2}=\beta\ \mbox{(say)}$ with $\beta>0$, then a lower bound for the  probability of blow-up solution of the system \eqref{b1}-\eqref{b2} is given by
    \begin{align}
    \mathbb{P}\{ \tau<\infty\}\geq \int^{\infty}_{\frac{\rho^{2}2^{\beta-1} E^{-\beta}(0)}{\beta}} h_{1}(y)dy, \nonumber 
    \end{align}
    where $\tau \leq \tau_{1}^{\ast}$, $\tau_{1}^{\ast}$ is given in \eqref{ST1}, 
    \begin{align}
    h_{1}(y)=\frac{(2a/(\rho^{2}y))^{(2a /\rho^{2})}}{y \Gamma(2 a /\rho^{2})} \exp\left(-\frac{2a}{\rho^{2}y}\right),\ \mbox{and}\ a=\beta(\lambda-\gamma+k^{2}).  \nonumber 
    \end{align} 
    \item[2.] If $\beta_{1}>\beta_{2}$, then a lower bound for the probability of blow-up solution of the system \eqref{b1}-\eqref{b2} is given by 
    \begin{align}
    \mathbb{P}\{ \tau<\infty\}\geq \int^{\infty}_{\frac{\rho^{2} N}{2}} h_{2}(y)dy, \nonumber 
    \end{align}
    where $\tau \leq \tau_{2}^{\ast}$, $\tau_{2}^{\ast}$ is given in \eqref{ST2}
    \begin{align}
    h_{2}(y)=\frac{(2a_{1}/(\rho^{2}y))^{(2a_{1} /\rho^{2})}}{y \Gamma(2 a_{1} /\rho^{2})} \exp\left(-\frac{2a_{1}}{\rho^{2}y}\right),\  a_{1}=\beta_{2}(\lambda-\gamma+k^{2}),\  k^{2}=\max\left\lbrace  \frac{k_{11}^{2}}{2}, \frac{k_{21}^{2}}{2} \right\rbrace,  \nonumber 
    \end{align}$\gamma=\min\{\gamma_1,\gamma_2\}$, $N=\left[\beta_{2} E^{\beta_{2}}(0)\left(\frac{\epsilon_0}{2^{1+\beta_{2}}}-\frac{ \epsilon_{0}^{\frac{1+\beta_{1}}{\beta_{1}-\beta_{2}}}D_{1}}{E^{1+\beta_{2}}(0)}\right)\right]^{-1}$ \mbox{ and } $E(0)=\displaystyle\int_{D}[f_{1}(x)+f_{2}(x)]\psi(x)dx.$
    \end{itemize}
	\end{theorem}
	\begin{proof}
	\textbf{Case 1:} If $\beta_{1}=\beta_{2}=\beta,$(say), then by case 1 of Theorem \ref{thm4.1}, we have $\tau \leq \tau_{1}^{\ast},$ where $\tau_{1}^{\ast}$ is given in \eqref{ST1}. By definition of $\tau_{1}^{\ast},$ we have
	\begin{align}
	\mathbb{P}(\tau^{\ast}_{1}=\infty)&=\mathbb{P}\left(\int_{0}^{t} \exp\{\rho W(s)+\rho B^{H}(s)-as\} ds < 2^{\beta} {\beta}^{-1}E^{-\beta}(0),\ \mbox{for all}\ t \geq 0 \right) \nonumber\\
	&=\mathbb{P}\left(\int_{0}^{\infty} \exp\{\rho W(s)+\rho B^{H}(s)-as\} ds \leq 2^{\beta} {\beta}^{-1}E^{-\beta}(0) \right) \nonumber\\
	&=\mathbb{P}\left(\int_{0}^{\infty} \exp\{\rho \widetilde{B}(s)-as\} ds \leq 2^{\beta} {\beta}^{-1}E^{-\beta}(0) \right) \nonumber\\
	&=\mathbb{P}\left(\int_{0}^{\infty} \exp \left\lbrace 2 \widetilde{B}\left( \frac{\rho^{2}s}{4}\right) -as\right\rbrace  ds \leq 2^{\beta} {\beta}^{-1}E^{-\beta}(0) \right). \nonumber
	\end{align}
 By performing the transformation $t=\frac{\rho^2s}{4}$ and setting $\nu = \frac{2a}{\rho^{2}}$,  we get
	\begin{align}
		\mathbb{P}\{ \tau<\infty\} \geq \mathbb{P}\{ \tau^{\ast }_{1}<\infty\}=1-\mathbb{P}\{ \tau^{\ast }_{1}=\infty\}
		&=\mathbb{P}\Bigg(\frac{4}{\rho^{2}} \int_{0}^{\infty} \exp\{2 \widetilde{B}_{t}^{(\nu)}\}dt > 2^{\beta} {\beta}^{-1}E^{-\beta}(0) \Bigg),\nonumber
	\end{align}
	where $\widetilde{B}_{s}^{(\nu)}:=\widetilde{B}(s)-\nu s$ .
	It follows from \cite[Chapter 6, Corollary 1.2]{yor2001} that 
	\begin{equation}{\label{ex1}}
		\int_{0}^{\infty} \exp\{ 2\widetilde{B}_{t}^{(\nu)}\}dt \overset{Law}{=} \frac{1}{2Z_{\nu}},
	\end{equation}
	where $Z_{\nu}$ is a Gamma random variable with parameter $\nu,$ that is $\mathbb{P} (Z_{\nu} \in dy)= \displaystyle\frac{1}{\Gamma(\nu)}e^{-y} y^{\nu-1} dy.$ Therefore, we get 
	\begin{align}
		\mathbb{P}\{ \tau<\infty\}\geq \int^{\infty}_{\frac{\rho^{2}2^{\beta-1} E^{-\beta}(0)}{\beta}} h_{1}(y)dy, \nonumber 
	\end{align}
	where 
	\begin{align}
		h_{1}(y)=\frac{(2a/(\rho^{2}y))^{(2a /\rho^{2})}}{y \Gamma(2 a /\rho^{2})} \exp\left(-\frac{2a}{\rho^{2}y}\right).  \nonumber 
	\end{align}
	\noindent\textbf{{Case 2.}} If $\beta_{1}>\beta_{2}$, then by case 2 of Theorem \ref{thm4.1}, we have $\tau \leq \tau_{2}^{\ast},$ where $\tau_{2}^{\ast}$ is given in \eqref{ST2}. By the definition of $\tau_{2}^{\ast},$ we have
	\begin{align}
	&\mathbb{P}(\tau^{\ast}_{1}=\infty)\nonumber\\&=\mathbb{P}\left(\int_{0}^{t} \exp\{\rho \widetilde{B}(s)-a_{1}s\} ds <\left[\beta_{2} E^{\beta_{2}}(0)\left(\frac{\epsilon_0}{2^{1+\beta_{2}}}-\frac{ \epsilon_{0}^{\frac{1+\beta_{1}}{\beta_{1}-\beta_{2}}}D_{1}}{E^{1+\beta_{2}}(0)}\right)\right]^{-1} ,\ \mbox{for all}\ t \geq 0 \right) \nonumber\\
	&=\mathbb{P}\left(\int_{0}^{\infty} \exp\{\rho \widetilde{B}(s)-a_{1}s\} ds <\left[\beta_{2} E^{\beta_{2}}(0)\left(\frac{\epsilon_0}{2^{1+\beta_{2}}}-\frac{ \epsilon_{0}^{\frac{1+\beta_{1}}{\beta_{1}-\beta_{2}}}D_{1}}{E^{1+\beta_{2}}(0)}\right)\right]^{-1} \right). \nonumber
	\end{align}
	By using the procedure as in the case 1 of Theorem \ref{thm6.6}, we have
    \begin{align}
	\mathbb{P}\{ \tau<\infty\} &\geq \mathbb{P}\{ \tau^{\ast }_{2}<\infty\}=1-\mathbb{P}\{ \tau^{\ast }_{2}=\infty\}\nonumber\\
	&=\mathbb{P}\Bigg(\frac{4}{\rho^{2}} \int_{0}^{\infty} \exp\{2 \widetilde{B}_{t}^{(\nu_{1})}\}dt > \left[\beta_{2} E^{\beta_{2}}(0)\left(\frac{\epsilon_0}{2^{1+\beta_{2}}}-\frac{ \epsilon_{0}^{\frac{1+\beta_{1}}{\beta_{1}-\beta_{2}}}D_{1}}{E^{1+\beta_{2}}(0)}\right)\right]^{-1} \Bigg),\nonumber
    \end{align}
    where $\nu_{1}=\displaystyle\frac{2a_{1}}{\rho^{2}}$. Therefore from \eqref{ex1}, we obtain
    \begin{align}
    	\mathbb{P}\{ \tau<\infty\}\geq \int^{\infty}_{\frac{\rho^{2} N}{2}} h_{2}(y)dy, \nonumber 
    \end{align}
    where 
    \begin{align}
    	h_{2}(y)=\frac{(2a_{1}/\rho^{2}y)^{(2a_{1} /\rho^{2})}}{y \Gamma(2 a_{1} /\rho^{2})} \exp\left(-\frac{2a_{1}}{\rho^{2}y}\right),  \nonumber 
    \end{align}
    which completes the proof.
	\end{proof}	

     \medskip\noindent
     {\bf Acknowledgements:} The first author is supported by the University Research Fellowship of Periyar University, India. M. T. Mohan would  like to thank the Department of Science and Technology (DST), India for Innovation in Science Pursuit for Inspired Research (INSPIRE) Faculty Award (IFA17-MA110). The third author is supported by the Fund for Improvement of Science and Technology Infrastructure (FIST) of DST, New Delhi (SR/FST/MSI-115/2016).
      The authors sincerely would like to thank the associate editor and anonymous reviewers for their valuable comments and suggestions, which helped us to improve the quality of this manuscript significantly.\\
     
     \noindent {\bf Data availability:} Data sharing not applicable to this article as no datasets were generated or analysed during the current study.\\

     \noindent {\bf Disclosure statement:} No potential competing interest was reported by the authors.


\end{document}